\documentclass[12pt]{amsart}
\usepackage{amsmath}
\usepackage{amssymb}
\usepackage{amsfonts}
\usepackage{amsthm}
\usepackage{amscd}
\usepackage{relsize}
\usepackage{setspace}
\usepackage{geometry}
\usepackage{enumerate}
\usepackage{url}
\usepackage{xspace}
\usepackage{mathrsfs}
\usepackage{dsfont}
\usepackage{lmodern}

\usepackage{hyperref}

\newtheorem{theorem}{Theorem}[section]
\newtheorem{lemma}[theorem]{Lemma}
\newtheorem{prop}[theorem]{Proposition}
\newtheorem{corollary}[theorem]{Corollary}

\theoremstyle{definition}
\newtheorem{remark}[theorem]{Remark}

\newtheorem{definition}[theorem]{Definition}
\newtheorem{question}[theorem]{Question}

\newcommand{\triv}{\{1\}}

\newcommand{\QZ}{\mathrm{QZ}}
\newcommand{\CC}{\mathrm{C}}
\newcommand{\N}{\mathrm{N}}
\newcommand{\Z}{\mathrm{Z}}
\newcommand{\tdlc}{t.d.l.c.\@\xspace}

\newcommand{\cG}{\mathcal{G}}
\newcommand{\cH}{\mathcal{H}}
\newcommand{\cL}{\mathcal{L}}
\newcommand{\cU}{\mathcal{U}}
\newcommand{\ldlat}{\mathcal{LD}}

\newcommand{\lnorm}{\mathcal{LN}}
\newcommand{\lcent}{\mathcal{LC}}
\newcommand{\Res}{\mathrm{Res}}

\newcommand{\defbold}{\textbf}

\newcommand{\bC}{\mathbb{C}}
\newcommand{\bF}{\mathbb{F}}
\newcommand{\bN}{\mathbb{N}}

\newcommand{\bR}{\mathbb{R}}
\newcommand{\bZ}{\mathbb{Z}}

\newcommand{\mc}{\mathcal}

\newcommand{\ms}{\mathscr}

\newcommand{\g}{\mathfrak{g}}
\newcommand{\G}{\mathfrak{G}}
\newcommand{\hh}{\mathfrak{h}}

\newcommand{\WW}{\mathcal{W}}

\newcommand{\inv}{^{-1}}

\newcommand{\la}{\langle}
\newcommand{\ra}{\rangle}
\newcommand{\co}{\colon\thinspace}

\newcommand{\grp}[1]{\langle #1 \rangle}
\newcommand{\ol}[1]{\overline{#1}}

\DeclareMathOperator{\CAT}{CAT(0)}
\DeclareMathOperator{\Aut}{Aut}

\DeclareMathOperator{\con}{con}
\DeclareMathOperator{\GL}{GL}

\DeclareMathOperator{\height}{ht}

\DeclareMathOperator{\Kac}{Kac}
\DeclareMathOperator{\Pc}{Pc}
\DeclareMathOperator{\proj}{proj}

\DeclareMathOperator{\supp}{supp}
\DeclareMathOperator{\Hom}{Hom}
\DeclareMathOperator{\sph}{sph}

\begin{document}

\title[Locally normal subgroups and ends of Kac--Moody groups]{Locally normal subgroups and ends\\ 
of locally compact Kac--Moody groups}

\author{Pierre-Emmanuel \textsc{Caprace}}
\author{Timoth\'{e}e \textsc{Marquis}}
\author{Colin D. \textsc{Reid}}

\thanks{TM is a F.R.S.-FNRS Research associate; PEC and TM are supported in part by the FWO and the F.R.S.-FNRS under the EOS programme (project ID 40007542); CDR: Research supported in part by ARC grant FL170100032.}

\begin{abstract}
A locally normal subgroup in a topological group is a subgroup whose normaliser is open. 
In this paper, we provide a detailed description of the large-scale structure of closed locally normal subgroups of complete Kac--Moody groups over finite fields.  
Combining that description with the main result from \cite{Growingtrees}, we  show that under mild assumptions, if the Kac--Moody group is one-ended (a property that is easily determined from the generalised Cartan matrix), then it is locally indecomposable, which means that no open subgroup decomposes as a nontrivial direct product. 
\end{abstract}

\maketitle

\section{Introduction}

The general structure theory of totally disconnected locally compact (\tdlc) groups has developed dramatically over the last decades. The emphasis has been on those \tdlc groups that are compactly generated (any \tdlc group being a direct limit of such groups). One key aspect in this line of research is the decomposition of a compactly generated \tdlc group into simpler pieces; in particular, the paper \cite{CaMo11} showcases the special role, in the study of general \tdlc groups, of the class $\ms{S}$ of compactly generated \tdlc groups that are topologically simple and nondiscrete. Another key aspect is the study of \tdlc groups via their local structure, namely, properties evident in an arbitrarily small neighbourhood of the identity. Fundamental examples of local invariants are the structure, centraliser, and decomposition lattices introduced in \cite{CRWpart1}, and further studied in \cite{CRWpart2}, which we now briefly review. 

Let $G$ be a \tdlc group. Given closed subgroups $H, H' \leq G$, we write $H\sim_oH'$ if $H\cap H'$ is open in both $H$ and $H'$.  This is an equivalence relation, called \textbf{local equivalence}. A subgroup $H$ of $G$ is {\bf locally normal} if its normaliser is open. The {\bf structure lattice} of $G$ is the poset $\lnorm(G) = \mathrm{LN}(G)/\sim_o$, where $\mathrm{LN}(G)$ is the set of closed locally normal subgroups of $G$ ordered by inclusion. Let us now suppose that $G \in \ms S$. Then the sub-poset of $\lnorm(G)$ defined by
$$\lcent(G) := \mathrm{LC}(G)/\sim_o, \quad\textrm{where}\quad\mathrm{LC}(G) := \{ \CC_G(K) \ | \ K \in \mathrm{LN}(G)\},$$
admits a Boolean algebra structure, and is called the {\bf centraliser lattice} of $G$. Finally, there is also a subalgebra $\ldlat(G)$ of $\lcent(G)$, called the {\bf decomposition lattice} of $G$, consisting of those elements of $\lnorm(G)$ represented by direct factors of open subgroups. 

The centraliser (resp. decomposition) lattice always contains at least two local equivalence classes, denoted by $0$ and $\infty$, respectively defined as the class of the trivial subgroup and that of open subgroups of $G$. If $\lcent(G)=\{0, \infty\}$ (resp. $\ldlat(G) = \{0,\infty\}$), we say that it is  {\bf trivial}. This leads to a simple dichotomy within the class $\ms{S}$: either $G\in\ms{S}$ has nontrivial centraliser lattice, or it does not. In the former case, $G$ admits a continuous and faithful action by homeomorphims on a compact space $\Omega_G$ (namely, the Stone dual of the Boolean algebra $\lcent(G)$) with nice dynamical properties; this provides a powerful uniform tool to study such groups. In the latter case, the space $\Omega_G$ is reduced to a singleton and no substitute for that  tool is currently available. Our knowledge of the class of groups in $\ms{S}$ with a trivial centraliser lattice is quite limited at present: in fact, amongst the known examples of groups in $\ms{S}$, only those of ``Lie-theoretic origin'' are known or conjectured to have trivial centraliser lattice, namely, the simple algebraic groups over local fields (known) and the \emph{locally compact Kac--Moody groups} (conjectured). 

Locally compact Kac--Moody groups are obtained as completions of minimal Kac--Moody groups (as introduced in \cite{Tits87}) over a finite field. There are a few different ways to construct such completions;  for definiteness we specialise here to the complete \emph{geometric} Kac--Moody groups over finite fields: these are defined by letting the minimal Kac--Moody group act on its positive building, and then taking the closure in the permutation topology (see \cite[Chapter~8]{KMGbook}).

Locally compact Kac--Moody groups have peculiar properties that distinguish them in the class $\ms{S}$. Some locally compact Kac--Moody groups enjoy Kazhdan's property (T) (see \cite{DymaraJanusz}), and every known group in $\ms{S}$ with Kazhdan's property is either a simple algebraic group or a Kac--Moody group. It is also known that products of certain pairs of locally compact Kac--Moody groups contain irreducible non-uniform lattices, and that some of these lattices are simple groups (see \cite{CaRe}).

The first goal of this paper is to get a better understanding of the locally normal subgroups of locally compact Kac--Moody groups (see Theorem~\ref{intro:KM_loc_norm}), building on the results on their open subgroups established in \cite{openKM}. Proving that a group in $\ms{S}$ --- beyond the simple algebraic groups --- has trivial decomposition lattice (and, \emph{a fortiori}, trivial centraliser lattice) is notoriously difficult, and the second goal of this paper is to show, using a criterion we recently obtained in \cite{Growingtrees}, that (a large class of) locally compact Kac--Moody groups are indeed {\bf locally indecomposable}, that is, have trivial decomposition lattice $\ldlat(G)$ or, equivalently, no open subgroup decomposes as a nontrivial direct product (see Theorem~\ref{intro:KM_loc_indec}). 

 We now describe our main results in more details, refering to the preliminary sections \ref{sec:prelim} and \ref{sec:prelim_KM} for precise definitions and terminology.

Let $G$ be a complete geometric Kac--Moody group $G$ over a finite field.  
Consider first the poset $\widetilde{\mc{O}}(G) = \mc{O}(G)/\sim_f$, where $\mc{O}(G)$ is the class of open subgroups of $G$ ordered by inclusion and $O_1 \sim_f O_2$ if the indices $[O_1: O_1 \cap O_2]$ and $[O_2 : O_1 \cap O_2]$ are both finite.  In the present context, $\widetilde{\mc{O}}(G)$ was described up to conjugacy in $G$ by \cite[Theorem~3.3]{openKM}.  Indeed we can give a sharper description of exactly how conjugation in $G$ interacts with the partial order, via a more general result (Theorem~\ref{thm:building_commensurate}) about parabolic subgroups of groups with BN-pairs.

\begin{theorem}\label{intro:KM_open}
Let $G$ be a complete geometric Kac--Moody group over a finite field and let $(\WW,S)$ be its Weyl group.   Let $\Lambda_{(\WW,S)}$ be the set of standard parabolic subgroups of $\WW$ ordered by inclusion, and let $\Lambda^{\infty}_{(\WW,S)} = \Lambda_{(\WW,S)}/\sim_f$.  Then there is a unique map $\lambda: \mc{O}(G) \rightarrow \Lambda^{\infty}_{(\WW,S)}$, with the following properties:
\begin{enumerate}[(i)]
\item For all $J \subseteq S$, we have $\lambda(P_J) = [\WW_J]$, where $P_J$ is the standard parabolic subgroup of $G$ of type $J$ and $[\WW_J]$ denotes the $\sim_f$-class of $\WW_J$;
\item Given $H,K \in \mc{O}(G)$, we have $\lambda(H) = \lambda(K)$ if and only if there is $g \in G$ such that $gHg\inv \sim_f K$;
\item Given $H,K \in \mc{O}(G)$, we have $\lambda(H) < \lambda(K)$ if and only if there is $g \in G$ such that $gHg\inv \sim_f L$ for some subgroup $L$ of $K$ of infinite index.
\end{enumerate}
\end{theorem}

The poset $\Lambda^{\infty}_{(\WW,S)}$ of standard parabolics modulo finite index can be identified in an obvious way with the poset of essential subsets of $S$ ordered by inclusion (see Lemma~\ref{lemma:prelim_parabolic} and Corollary~\ref{cor:parabolic_inclusion}). A subset $J \subseteq S$ is called \textbf{essential} if each irreducible component of the Coxeter system $(W_J, J)$ is non-spherical. In particular, given the classification of finite Coxeter groups, $\Lambda^{\infty}_{(\WW,S)}$ is a finite poset that can easily be constructed from the Coxeter diagram.

Now consider the poset $\mathrm{LN}(G)$.  Given $H \in \mathrm{LN}(G)$, for a sufficiently small compact open subgroup $U$ of $G$, then $O = HU$ is an open subgroup of $G$, which only depends on $U$ up to finite index; thus we have a map 
\[
\theta: \mathrm{LN}(G) \rightarrow \widetilde{\mc{O}}(G); \; H \mapsto [HU].
\]
On the other hand, if $O = HU$ for a compact open subgroup $U$ of $\N_G(H)$, then $\Res(O) \le H$, where $\Res(O)$ is the intersection of all open normal subgroups of $O$.  The group $\Res(O)$  is a closed normal subgroup of $O$, hence an element of $\mathrm{LN}(G)$. Moreover $\Res(O)$ is not sensitive to changes of finite index.  So we have another map
\[
\phi: \widetilde{\mc{O}}(G) \rightarrow  \mathrm{LN}(G); \; [O] \mapsto \Res(O),
\]
such that $\phi\theta(H) \unlhd H$ for all $H \in \mathrm{LN}(G)$.  The strongest large-scale connection between $\widetilde{\mc{O}}(G)$ and $\mathrm{LN}(G)$ that we can reasonably hope for in this context is that $\Res(O)$ is cocompact in $O$ for all $O \in \mc{O}(G)$, which is equivalent to asking that $\theta\phi = \mathrm{id}_{\widetilde{\mc{O}}(G)}$, or equivalently, $H/\phi\theta(H)$ is compact for all $H \in \mathrm{LN}(G)$.  This is exactly what happens for complete geometric Kac--Moody groups over finite fields.

\begin{theorem}\label{intro:KM_loc_norm}
Let $G$ be a complete geometric Kac--Moody group over a finite field and let $H$ be a noncompact closed locally normal subgroup of $G$.  Then there exist some $g\in G$, some essential subset $J\subseteq S$ and some spherical subset $J'\subseteq J^\perp$ such that
\[
\Res(P_{J \cup J'}) \le gHg\inv \le P_{J\cup J'}.
\]
Moreover, $\Res(P_{J \cup J'})$ is cocompact in $P_{J \cup J'}$.
\end{theorem}

In fact we obtain a stronger result involving contraction groups and the Levi decomposition; see Theorem~\ref{thm:struct_ncpct_lnsbgr}.

Using the constraints we have on closed locally normal subgroups of $G$, we can apply the criterion \cite[Corollary~4.9]{Growingtrees} to conclude the following.

\begin{theorem}\label{intro:KM_loc_indec}
Let $G$ be a complete geometric Kac--Moody group over a finite field.  Suppose $G$ is one-ended and locally finitely generated.  Then $G$ is locally indecomposable.
\end{theorem}

We see that $G$ is one-ended if and only if its Weyl group is one-ended, and the latter has a straightforward characterisation due to M. Davis; see Section~\ref{subsec:ends} for details.

For the question of whether $G$ is locally finitely generated, the answer is known to be positive in many but not all cases.  The following corollary employs some known sufficient conditions.

\begin{corollary}\label{corintro:localindecKM}
Let $G$ be a complete geometric Kac--Moody group over a finite field of order $q$ and characteristic $p$, with indecomposable generalised Cartan matrix $A = (a_{ij})$ and Weyl group $\WW_A$, and write $M_A := \max_{i \neq j}|a_{ij}|$. Suppose that at least one of the following conditions holds:
\begin{enumerate}[(i)]
\item $\WW_A$ is one-ended and $p > M_A$. 
\item $A$ is $2$-spherical.  Moreover, $q \ge 3$ if $M_A = 2$ and $q \ge 4$ if $M_A = 3$.  
\end{enumerate}
Then $G$ is locally indecomposable.
\end{corollary}

\subsection{Structure of the article}

The first sections contain the relevant preliminaries on \tdlc groups (Section~\ref{sec:prelim}) and BN-pairs and Kac--Moody groups (Section~\ref{sec:prelim_KM}). In Section~\ref{sec:conjugacy_parabolic} we prove a result about conjugacy classes of parabolic subgroups of a group with a BN-pair, which is used to prove Theorem~\ref{intro:KM_open}.  In Section~\ref{sec:mainKM}, we focus on complete geometric Kac--Moody groups over finite fields and describe their locally normal subgroups, number of ends, and the relationship between them, using the results of the previous sections where relevant.

\section{Preliminaries on \tdlc groups}\label{sec:prelim}
We start by introducing the terminology required to state the local indecomposability criterion \cite[Corollary~4.9]{Growingtrees}, and recall some properties of contraction groups that will be used later.

\subsection{Locally indecomposable groups}

Let $G$ be a \tdlc group.  A subgroup of $G$ is called \defbold{locally normal} if its normaliser is open.  A closed subgroup $K$ of $G$ is a \defbold{local direct factor} if it is a direct factor of some open subgroup $O$ of $G$, that is, $O$ splits, as a topological group, as a direct product $O = K \times L$.  Note that every local direct factor is locally normal.  The group $G$ is \defbold{locally indecomposable} if every local direct factor of $G$ is discrete or open.

Being locally indecomposable is a local property:
\begin{lemma}\label{lem:loc_indecomp}
Let $G$ be a \tdlc group and let $U$ be an open subgroup of $G$.  Then $G$ is locally indecomposable if and only if $U$ is locally indecomposable.
\end{lemma}

\begin{proof}
We prove the contrapositive in both directions: that is, $G$ has a local direct factor that is neither discrete nor open if and only if $U$ has such a local direct factor.  Indeed, if $G$ has an open subgroup $O = K \times L$ where $K$ is neither discrete nor open, then since $O$ carries the product topology, the group $(K \cap U) \times (L \cap U)$ is open in $U$; we then see that $K \cap U$ is a local direct factor of $U$ that is neither discrete nor open.  Conversely if $K$ is a local direct factor of $U$ that is neither discrete nor open, then $K$ is also a local direct factor of $G$.
\end{proof}

\subsection{Local finiteness properties}

\begin{definition}
A profinite group is \defbold{topologically finitely generated} if it has a finitely generated dense subgroup. A \tdlc group $G$ is \defbold{locally finitely generated} if every compact open subgroup of $G$ is topologically finitely generated.

A profinite group $U$ is of \defbold{finite quotient type} if for each natural number $n$, there are only finitely many open normal subgroups of $U$ of index $n$.  A \tdlc group $G$ is \defbold{locally of finite quotient type} if every compact open subgroup of $G$ is of finite quotient type.  
\end{definition}

The two notions are strongly related; moreover, to determine if $G$ is locally finitely generated, it suffices to consider a single compact open subgroup.

\begin{lemma}\label{lem:fqt_local}
Let $G$ be a \tdlc group, and let $U$ be a compact open subgroup.
\begin{enumerate}
\item 
If $U$ is topologically finitely generated, then it is of finite quotient type, and every compact open subgroup $V$ of $G$ is topologically finitely generated.
\item
If $U$ is of finite quotient type and pro-$p$ for some prime $p$, then it is topologically finitely generated.
\end{enumerate}
\end{lemma}

\begin{proof}
(1)
Suppose $U = \ol{F}$ where $F$ is a finitely generated subgroup, and let $d(F)$ be the number of generators of $F$.  Note that for every coset $uH$ of an open subgroup $H$ of $U$, we have $uH = \ol{F \cap uH}$.  Consequently, given $n \in \bN$ and open subgroups $H_1 \neq H_2$ of $U$ of index $n$, we see that $F \cap H_1$ and $F \cap H_2$ are distinct subgroups of $F$ of index $n$.  Since $F$ is finitely generated, it only has finitely many subgroups of index $n$.  Thus there are only finitely many open subgroups of $U$ of index $n$, so $U$ is of finite quotient type.

Given a compact open subgroup $V$ of $G$, write $W = U \cap V$.  We then have $|F:F \cap W| = |U:W| < \infty$, and then the Schreier index formula yields 
\[
d(F \cap W) \le 1 + |U:W|(d(F)-1) < \infty.
\]
In turn, there is a dense subgroup of $V$ generated by $F \cap W$ and $X$, where $X$ is a set of coset representatives for $W$ in $V$, so $|X| = |V:W| < \infty$.  Thus $V$ is topologically finitely generated.  

(2)
Suppose now that $U$ is of finite quotient type and pro-$p$. By \cite[Lemma 2.8.7(a)]{RZ10}, every maximal closed subgroup of $U$ has index $p$; there are thus only finitely many maximal closed subgroups of $U$, so the Frattini subgroup $\Phi(U)$ has finite index.  By \cite[Proposition 2.8.10]{RZ10} it follows that $U$ is topologically finitely generated. 
\end{proof}

\subsection{One-endedness}
Let $X$ be a metric space. A \defbold{geodesic segment} (resp. \defbold{ray}, \defbold{line}) in $X$ is an isometry $r\co I\to X$, where $I$ is a closed interval of $\bR$ (resp. $I=[0,\infty)$, $I=\bR$) and $\bR$ is equipped with the usual Euclidean metric. The space $X$ is \defbold{geodesic} if any two points of $X$ are connected by (the image of) a geodesic segment. It is {\bf proper} if every closed ball in $X$ is compact.

A geodesic ray $r$ is \defbold{proper} if ``$r(t)$ goes to infinity as $t\to\infty$'', that is, for any compact $C\subseteq X$ there is some $N\in\bN$ such that $r([N,\infty))\subseteq X\setminus C$. The space $X$ is \defbold{one-ended} if for any two proper rays $r_1,r_2\co[0,\infty)\to X$ and any compact subset $C\subseteq X$, there is some $N\in\bN$ such that $r_1([N,\infty))$ and $r_2([N,\infty))$ are contained in the same path component of $X\setminus C$.

Let $G$ be a compactly generated \tdlc group. Then $G$ acts {\bf geometrically} (that is, properly and cocompactly) by isometries on a proper geodesic metric space $X$ (see e.g. \cite[\S2.2]{Growingtrees} for more details and definitions). Moreover, any two such $X$'s are quasi-isometric. Since one-endedness is a quasi-isometric invariant of the space, it then makes sense to call $G$ \defbold{one-ended} if it acts geometrically by isometries on a one-ended proper geodesic metric space $X$.

\subsection{The local indecomposability criterion}

The \defbold{quasi-centre} $\QZ(G)$ of a \tdlc group $G$ consists of all elements of $G$ with open centraliser. 

Here is the announced criterion.

\begin{theorem}[{\cite[Corollary~4.9]{Growingtrees}}]\label{thm:cpctend:bis}
Let $G$ be a nontrivial compactly generated \tdlc group.  Suppose the following: $G$ is one-ended; $\QZ(G)=\triv$; $G$ has no nontrivial compact normal subgroups; $G$ is locally of finite quotient type; no open subgroup of $G$ has an infinite discrete quotient; and the centraliser of every nontrivial closed locally normal subgroup is compact.  Then $G$ is locally indecomposable.
\end{theorem}

\subsection{Contraction groups}

Given a topological group $G$ and an element $g\in G$, we define the \defbold{contraction group}
$$\con_G(g):=\{x\in G \ | \ g^nxg^{-n}\stackrel{n\to\infty}{\to}1\}.$$
The contraction group is invariant under replacing $g$ with a positive power (\cite[Lemma~2.8]{ReidFlat}).  We will drop the subscript $G$ when the ambient group is clear from context.

Note that if $K$ is a compact subset of $G$ contained in $\con_G(g)$, then $K$ is uniformly contracted by $g$: if $U$ is an identity neighbourhood, a compactness argument shows that $g^nKg^{-n} \subseteq U$ for all but finitely many $n \ge 0$.

The \defbold{Tits core} of a \tdlc group $G$ is its normal subgroup
$$ G^{\dagger} = \la \ol{\con(g)} \mid g \in G \ra.$$ 
Given a subset $H$ of $G$, the \defbold{relative Tits core} is
$$ G^{\dagger}_H = \la \ol{\con_G(h)} \mid h \in H \cup H^{\inv} \ra;$$
note that if $H$ is an open subgroup of $G$, then $\con_G(h) = \con_H(h)$ for all $h \in H$, so $G^\dagger_H = H^\dagger$.
For $g \in G$ we define
$$ G^{\dagger}_g := G^{\dagger}_{\la g \ra} =  \la \ol{\con_G(g)}, \ol{\con_G(g\inv)} \ra.$$

We recall some facts about relative Tits cores that will be useful later.

\begin{lemma}[{See \cite[Theorem~1.2(i)]{ReidFlat}}]\label{lemma:dagger_ln}
Let $G$ be a \tdlc group and let $g \in G$.  Then $G^{\dagger}_g$ is locally normal in $G$.
\end{lemma}

\begin{lemma}[{See \cite[Theorem~1.4]{ReidFlat}}]\label{lemma:dagger_cc}
Let $G$ be a \tdlc group, let $H$ be a closed subgroup of $G$ and let $K$ be a subgroup of $H$, such that $\overline{K}$ is cocompact in $H$.  Then $G^{\dagger}_H = G^{\dagger}_K$.
\end{lemma}

\begin{lemma}[{\cite[Theorem~1.5]{ReidFlat}}]\label{lemma:dagger_containment}
Let $G$ be a \tdlc group, let $D$ be a subgroup of $G$ (not necessarily closed), and let $X \subseteq \ol{D}$.  Suppose that there is an open subgroup $U$ of $G$ such that $U \cap G^\dagger_X \le \N_G(D)$.  Then $G^\dagger_X \le D$.
\end{lemma}

\begin{prop}\label{prop:dagger_gg}
Let $G$ be a \tdlc group and let $g \in G$.  Suppose that $\overline{G^\dagger_g}$ is cocompact in $G$.  Then $G^\dagger = G^\dagger_g$.
\end{prop}
\begin{proof}
The normaliser $O$ of $G^\dagger_g$ is open in $G$ by Lemma~\ref{lemma:dagger_ln}.  In particular, $O$ is closed, so it contains $H = \overline{G^\dagger_g}$; hence $O$ is cocompact, so it has finite index in $G$.  Hence $O^\dagger = G^\dagger$, and by Lemma~\ref{lemma:dagger_cc} (applied to $G=H:=O$ and $K:=H$), we have $O^\dagger = O^\dagger_H$.  Since $G^\dagger_g$ is normal in $O$, Lemma~\ref{lemma:dagger_containment} (applied to $G=U:=O$, $D:=G^\dagger_g$, and $X:=H$) now shows that $O^\dagger_H \leq G^\dagger_g$.  We have now shown $G^\dagger \leq G^\dagger_g$; the reverse inclusion is clear.
\end{proof}

\section{Preliminaries on Kac--Moody groups}\label{sec:prelim_KM}

\subsection{Coxeter groups}

The general reference for this subsection and the next two is \cite{BrownAbr}. 

A \defbold{Coxeter system} is a pair $(\WW,S)$ consisting of a group $\WW$ (the \defbold{Coxeter group}) with a specified subset $S = \{s_1,\dots,s_n\} \subseteq \WW$ and presentation
\[
\WW = \langle s_1,\dots,s_n \mid s^2_i \; (1 \le i \le n), \; (s_is_j)^{m_{ij}} \; (1 \le i < j \le n) \rangle
\]
where $2 \le m_{ij} \le \infty$ (here $(s_is_j)^{\infty}$ can be read as the absence of a relation).  For the purposes of this article we only allow Coxeter groups that are finitely generated, that is, $|S| < \infty$. 

For each subset $I$ of $S$ there is an associated \defbold{standard parabolic subgroup} $\WW_I:=\langle I\rangle$, which also forms a Coxeter system $(\WW_I,I)$; a \defbold{parabolic subgroup} of $\WW$ is a conjugate of a standard parabolic subgroup.

A Coxeter system $(\WW,S)$ has an associated diagram $\Gamma_S = \Gamma_{(\WW,S)}$, with vertex set $S$ and an edge between $s_i$ and $s_j$ if and only if $m_{ij} \ge 3$ (and labelled $m_{ij}$ if $m_{ij} \ge 4$); the connected components of this graph are the \defbold{components} of $S$ (or of $(\WW,S)$), and $(\WW,S)$ is \defbold{irreducible} if $\Gamma_S$ is connected. We will often identify a subset $J$ of $S$ with the induced subgraph of $\Gamma_S$ with vertex set $J$.  A subset $J$ of $S$ is \defbold{spherical} if it generates a finite subgroup; more generally we define the \defbold{spherical part} $J^{\sph}$ of $J \subseteq S$ to be the union of the spherical components of $\Gamma_J$. We also let $J^{\infty}:=J\setminus J^{\sph}$ denote the {\bf essential part} of $J$, and call $J$ {\bf essential} if $J=J^\infty$.  Finally, we set $J^\perp:=\{i\in I \ | \ \textrm{$m_{ij}=2$ for all $j\in J$}\}$.

A decomposition $w=s_{i_1}\dots s_{i_d}$ of an element $w\in\WW$ as a product of generators with $d\in\bN$ minimal is called {\bf reduced}; in that case, $d$ is called the {\bf length} of $w$, denoted $\ell(w)$. The set 
$$\supp(w):=\{i_1,\dots,i_d\}\subseteq I$$
is then independent of the choice of a reduced decomposition for $w$, and is called the {\bf support} of $w$.

The intersection $\Pc(w)$ of all parabolic subgroups containing an element $w\in\WW$ is again a parabolic subgroup, called the {\bf parabolic closure} of $w$.  We record for future reference the following fact about parabolic closures.

\begin{lemma}\label{lemma:prelim_parabolic}
Let $J\subseteq I$ be essential. Then $\WW_J$ has no proper parabolic subgroups of finite index.  There exists $w\in \WW$ such that $\Pc(w)=\WW_J$; moreover, for any such $w$, then $\Pc(w^n)=\WW_J$ for all $n \neq 0$.
\end{lemma}
\begin{proof}
The first conclusion follows from \cite[Proposition~2.43]{BrownAbr}.  The existence of $w\in \WW$ such that $\Pc(w)=\WW_J$ is given for instance by \cite[Corollary~4.3]{CF10}.  Given such an element $w$, since $\WW_J$ has no proper parabolic subgroups of finite index, it follows by \cite[Lemma~2.4]{openKM} that $\Pc(w^n)=\WW_J$ for all $n \neq 0$.
\end{proof}

Standard parabolics of a Coxeter system can be conjugate, but the conjugating element is necessarily of a special form given by V. Deodhar (\cite{Deo}); in particular, the essential part is preserved.  The following is an expanded version of \cite[Lemma~2.1]{openKM}.
 
\begin{lemma}\label{lem:Deo}
Let $(\WW,S)$ be a Coxeter system and let $J \subseteq S$ be essential.  Then the normaliser $\N_{\WW}(\WW_J) = \WW_J \times {\WW}_{J^\perp}$.  Moreover, if $w \in {\WW}$ is such that $w\inv Jw \subseteq S$, then $w \in {\WW}_{J^\perp}$; in particular, $w\inv Jw = J$ and $\CC_{\WW}({\WW}_J) = {\WW}_{J^\perp}$.
\end{lemma}

\begin{proof}
Let $w \in \WW$ such that $w\inv W_Jw=W_{J'}$ for some $J'\subseteq S$. By \cite[Proposition~3.1.6]{Kra09} (see also \cite[Proposition~1.3.5(a)]{Kra09}), we can write $w = xy$ where $x \in \WW_J$ and $y\inv\Pi_J = \Pi_{J'}$ (with $\Pi_J$ denoting, as in \emph{loc. cit.}, the set of basis vectors in the standard linear representation of $\WW_J$). Moreover, \cite[\S3.1]{Kra09} (or \cite[Proposition~5.5]{Deo}) yields sequences $J=I_0,I_1,\dots,I_{t+1}=J'$ of subsets of $S$ and $s_0,\dots,s_t$ of elements of $S$ such that for each $i$ the component $K_i$ of $J_i\cup\{s_i\}$ containing $s_i$ is spherical, and such that $y=\nu_0\dots\nu_t$ where $\nu_i:=\nu(I_i,s_i):=w_{K_i \setminus \{s_i\}} w_{K_i}$ satisfies $\nu_i\inv I_{i}\nu_i=I_{i+1}$ (here $w_K$ denotes the longest element of $\WW_K$). Since $I_0=J$ is essential, we have $K_0=\{s_0\}$ and hence $\nu_0=s_0\in J^\perp$ and $I_1=J$. Reasoning inductively, the same observation yields $I_i=J$ and $\nu_i=s_i\in J^\perp$ for all $i$, so that $J=J'$ and $y\in J^\perp$. Since the normaliser of $J$ in $\WW_J$ is trivial by \cite[Proposition~2.73]{BrownAbr}, the lemma now easily follows.
\end{proof}

\subsection{Coxeter complexes}
Given a Coxeter system $(\WW,S)$, the {\bf Coxeter complex} $\Sigma=\Sigma(\WW,S)$ is the simplicial complex with simplices the cosets $w\WW_J$ ($w\in\WW$, $J\subseteq S$), and face relation $\leq$ the opposite of the inclusion relation. The maximal simplices of $\Sigma$ (namely, the singletons $\{w\}$ with $w\in W$) are called {\bf chambers}. Given $w\in W$ and $s\in S$, the chambers $\{w\}$ and $\{ws\}$ are called {\bf $s$-adjacent}. A {\bf gallery} is a sequence $\Gamma=(D_0,\dots,D_k)$ of chambers such that for each $i$, the chambers $D_{i-1}$ and $D_i$ are $t_i$-adjacent for some $t_i\in S$; one then calls $\mathrm{typ}(\Gamma):=(t_1,\dots,t_k)\in S^k$ the {\bf type} of $\Gamma$ and $\ell(\Gamma):=k$ its {\bf length}. The {\bf chamber distance} $\mathrm{d_{Ch}}(C,D)$ between two chambers $C,D$ is the length of a minimal-length gallery connecting $C$ and $D$. 

Let $J\subseteq S$. A {\bf $J$-gallery} is a gallery $\Gamma$ with $\mathrm{typ}(\Gamma)\subseteq J^{\ell(\Gamma)}$. A {\bf $J$-residue} of $\Sigma$ (or {\bf residue of type $J$}) is the set of chambers connected to a given chamber by a $J$-gallery. If $C$ is a chamber and $R$ a residue, there is a unique chamber $C'$ of $R$ minimising $\mathrm{d_{Ch}}(C,C')$; it is called the {\bf projection} of $C$ on $R$ and is denoted $\mathrm{proj}_R(C)$. Two residues $R,R'$ are called {\bf parallel} if $\proj_R(R')=R$ and $\proj_{R'}(R)=R'$; equivalently, $\proj_R|_{R'}\co R'\to R$ is a bijection (whose inverse is then $\proj_{R'}|_R$).

The Coxeter complex $\Sigma$ admits a CAT(0) metric realisation $|\Sigma|_{\CAT}$, called the {\bf Davis complex} of $(\WW,S)$ (see \cite{Davis}). It is both a proper and geodesic metric space. The action by left translations of $\WW$ on $\Sigma$ induces an isometric action of $\WW$ on $|\Sigma|_{\CAT}$. In the sequel, we will identify $\Sigma$ with $|\Sigma|_{\CAT}$ and the chambers of $\Sigma$ with their metric realisation in $|\Sigma|_{\CAT}$.

\subsection{BN-pairs and buildings}\label{subsection:BNpairsaB}

A \defbold{BN-pair} in a group $G$ is a pair $(B,N)$ of subgroups such that $G = \grp{B,N}$ and $T:= B \cap N$ is normal in $N$, and $\WW = N/T$ admits a generating set $S$ such that the following holds:
\begin{enumerate}[(i)]
\item Given $s \in S$ and $w \in \WW$, then $BwB . BsB \subseteq BwB \cup BwsB$;
\item For all $s \in S$ we have $sBs\inv \not\subseteq B$.
\end{enumerate}
If such a set $S$ exists, it is uniquely determined by $B$ and $N$, and $(\WW,S)$ is a Coxeter system, with $\WW$ (or $(\WW,S)$) being the \defbold{Weyl group} of the BN-pair. Note that as $T\subseteq B$, for each $n\in N$, the cosets $nB$ and $Bn$ only depend on the image $w$ of $n$ in $\WW$, whence the slight abuse of notation $wB:= nB$  (resp. $Bw:=Bn$) in (i)--(ii) above.

We also recall that $G$ admits a double coset decomposition, the \defbold{Bruhat decomposition} $G = \bigsqcup_{w \in W} BwB$.  Analogously to subgroups of the Weyl group, subgroups of $G$ of the form $P_J = B\WW_JB$ are called \defbold{standard parabolic subgroups} of $G$ (or of the BN-pair); these are precisely the subgroups of $G$ that contain $B$.  A \defbold{parabolic subgroup} is then a conjugate of a standard parabolic subgroup.

Given a group $G$ with a BN-pair $(B,N)$, there is an associated {\bf building} $X$ of type $(\WW,S)$ on which $G$ acts by simplicial automorphisms: $X$ is a simplicial complex which is the union of copies of $\Sigma(\WW,S)$, called {\bf apartments}, and the $G$-action on the chambers (i.e. maximal simplicies) of $X$ can be identified with the left translation action of $G$ on the set $G/B$ of cosets of $B$ in $G$. In particular, $B$ is the stabiliser in $G$ of the {\bf fundamental chamber} $C_0:=B$ (the trivial coset). The subgroup $N$ stabilises a ``fundamental'' apartment $\Sigma_0$, and the $N$-action on $\Sigma_0\approx\Sigma(W,S)$ has kernel $T$ and can be identified with the natural $\WW$-action on $\Sigma(W,S)$.

The notions of galleries, chamber distance, (parallel) residues and projections naturally extend from the setting of Coxeter complexes to that of buildings. The standard parabolic subgroup $P_J$ of type $J$ then coincides with the stabiliser in $G$ of the {\bf standard $J$-residue} (i.e. the $J$-residue containing $C_0$), and hence parabolic subgroups of $G$ correspond to stabilisers of residues of $X$. The chamber distance can be refined as follows: the {\bf Weyl distance} $\delta(C,D)$ from a chamber $C$ to a chamber $D$ is the element of $\WW$ admitting the type of a minimal length gallery from $C$ to $D$ as an expression. If $R,R'$ are parallel residues, then the Weyl distance from a chamber $C$ of $R$ to its projection $\proj_{R'}(C)$ on $R'$ is independent of the choice of $C$, and is called the {\bf Weyl distance} from $R$ to $R'$.

Finally, any building $X$ of type $(\WW,S)$ admits a $\CAT$ metric realisation, called its {\bf Davis realisation} (see \cite{Davis}), such that the restriction of the $\CAT$ metric to each apartment yields an isometric copy of the Davis complex $|\Sigma(\WW,S)|_{\CAT}$. In particular, the Davis realisation of $X$ is a geodesic metric space, which is moreover proper when $X$ is {\bf locally finite}, that is, when $\{s\}$-residues are finite for all $s\in S$.

\subsection{Kac--Moody root systems}\label{subsubsection:KMRS}
The general reference for this subsection is \cite{Kac}. 
For the remainder of this section, we specialise to Kac--Moody groups, and fix a generalised Cartan matrix (GCM) $A=(a_{ij})_{i,j\in I}$ indexed by the finite set $I$, as well as a realisation $(\hh,\Pi,\Pi^{\vee})$ of $A$ in the sense of \cite[\S 1.1]{Kac}, with set of {\bf simple roots} $\Pi=\{\alpha_i \ | \ i\in I\}$, set of {\bf simple coroots} $\Pi^{\vee}=\{\alpha^{\vee}_i \ | \ i\in I\}$, and Cartan algebra $\hh=\Lambda^{\vee}\otimes_{\bZ}\bC$.

We can then consider the {\bf Kac--Moody algebra} $\g(A)$ associated with $A$; see \cite[\S 1.2--1.3]{Kac} for the construction.  If $A$ is a Cartan matrix, $\g(A)$ is the corresponding semisimple Lie algebra; otherwise, $\g(A)$ is an infinite-dimensional complex Lie algebra defined by the Serre presentation associated to $A$, in the same manner as a semisimple Lie algebra can be reconstructed from its Cartan matrix.  The Kac--Moody algebra $\g(A)$ admits a root space decomposition $\g(A)=\hh\oplus\bigoplus_{\alpha\in\Delta}\g_{\alpha}$ with respect to the adjoint action of the Cartan subalgebra $\hh$, with corresponding root spaces $$\g_{\alpha}:=\{x\in\g(A) \ | \ [h,x]=\alpha(h)x \ \forall h\in\hh\}$$ and root system $\Delta:=\{\alpha\in\hh^*\setminus\{0\} \ | \ \g_{\alpha}\neq\{0\}\}$. 

Set $Q_+:=\bigoplus_{i\in I}\bN\alpha_i$. We let $\Delta^+:=\Delta\cap Q_+$ denote the set of {\bf positive roots}, so that $\Delta=\Delta^+\dot{\cup}(-\Delta^+)$. The {\bf height} of a root $\alpha=\pm\sum_{i\in I}n_i\alpha_i$ is $\height(\alpha):=\pm\sum_{i\in I}n_i\in\bZ$. We also define the {\bf support} 
$$\supp(\alpha):=\{i\in I \ | \ n_i\neq 0\}\subseteq I$$
of $\alpha$. 
For a subset $J\subseteq I$, we set
$$\Delta(J):=\Delta\cap \bigoplus_{i\in J}\bZ\alpha_i, \quad \Delta^+(J):=\Delta^+\cap\Delta(J), \quad \textrm{and}\quad \Delta^+_J:=\Delta^+\setminus\Delta^+(J).$$

The {\bf Weyl group} $\WW=\WW_A$ of $A$ is the subgroup of $\GL(\hh^*)$ generated by the {\bf simple reflections} $s_i$ ($i\in I$) defined by
$$s_i\co\hh^*\to\hh^*: \alpha\mapsto \alpha-\la\alpha,\alpha_i^{\vee}\ra\alpha_i.$$
The pair $(\WW,S:=\{s_i \ | \ i\in I\})$ is then a Coxeter system (when convenient, we will also identify a subset $J\subseteq I$ with $\{s_i \ | \ i\in J\}\subseteq S$).  The parameters $m_{ij}$ ($i \neq j$) of the Coxeter presentation are given as follows:
\[
\begin{tabular} { c | c c c c c }
$a_{ij}a_{ji}$ & $0$ & $1$ & $2$ & $3$ &$\ge 4$ \\
\hline 
$m_{ij}$ & $2$ & $3$ & $4$ & $6$ & $\infty$
\end{tabular}
\]
We call $A$ {\bf indecomposable} if $\WW$ is irreducible. The indecomposable GCM $A$ can be of {\bf finite}, {\bf affine} or {\bf indefinite} type (see \cite[Chapter~4]{Kac}).  Accordingly, we call a subset $J\subseteq I$ of finite/affine/indefinite type if the GCM $A_J:=(a_{ij})_{i,j\in J}$ is of that type.

The Weyl group $\WW$ stabilises $\Delta$. One then defines the set of {\bf real} roots as $$\Delta^{re}:=\WW.\{\alpha_i \ | \ i\in I\}\subseteq\Delta.$$
To each $\alpha=w\alpha_i\in\Delta^{re}$ ($w\in\WW$, $i\in I$), one associates the {\bf reflection} $r_{\alpha}:=ws_iw\inv\in\WW$, which depends only on $\alpha$. 
One also sets $\Delta^{re+}:=\Delta^{re}\cap\Delta^+$ and, for any subset $J\subseteq I$, $$\Delta^{re}(J):=\Delta^{re}\cap\Delta(J),\quad \Delta^{re+}(J):=\Delta^{re+}\cap\Delta(J)\quad\textrm{and}\quad \Delta^{re+}_J:=\Delta^{re+}\cap\Delta_J.$$

For each $\alpha\in\Delta^{re}$, the fixed point set $\partial\alpha$ of $r_{\alpha}\in\WW$ in $\Sigma=\Sigma(\WW,S)$ is called a {\bf wall}, and we set $r_{\partial\alpha}:=r_{\alpha}$. The subspace $\Sigma\setminus\partial\alpha$ has two connected components, called {\bf half-spaces}, which we denote by $\alpha$ and $-\alpha$: this provides a $\WW$-equivariant identification between $\Delta^{re}$ and the set of half-spaces of $\Sigma$. Any infinite order element $w\in\WW$ acts on $\Sigma$ as a hyperbolic isometry, i.e. it acts by translations on some geodesic line (a {\bf $w$-axis}). Walls of $\Sigma$ are connected in the following strong sense: if $m$ is a wall intersecting a geodesic line $L$, then either $|L\cap m|=1$ (one says $m$ is {\bf transverse} to $L$) or $L\subseteq m$. A root $\alpha\in\Delta^{re}$ is {\bf $w$-essential} if $\partial\alpha$ is transverse to some (equivalently, any) $w$-axis.

\subsection{Kac--Moody groups}\label{subsubsection:KMG}

The general reference for this subsection is \cite{KMGbook}. Throughout, we fix a finite field $k=\bF_q$ of order $q$ and characteristic $p$. We also fix a Kac--Moody root datum 
\[
\mc{D}=(I,A,\Lambda,(c_i)_{i\in I},(h_i)_{i\in I})
\]
with GCM $A$. Thus, $\Lambda$ is a free $\bZ$-module containing the $c_i$'s, its $\bZ$-dual $\Lambda^{\vee}$ contains the $h_i$'s, and $\la c_j,h_i\ra=a_{ij}$ for all $i,j\in I$. The readers unfamiliar with Kac--Moody root data may safely assume that $\mc{D}=\mc{D}^A_{\Kac}$ (see \cite[Example~7.10]{KMGbook}), in which case $\mc{D}$ simply encodes the realisation $(\hh,\Pi,\Pi^{\vee})$ of $A$, where $\hh=\Lambda^{\vee}\otimes_{\bZ}\bC$, $\alpha_i=c_i\in\hh^*$ and $\alpha_i^{\vee}=h_i\in\hh$.

Let $\G_{\mc{D}}$ be the constructive Tits functor of type $\mc{D}$ introduced by Tits (\cite{Tits87}), and let $\cG:=\G_{\mc{D}}(k)$ be the corresponding {\bf minimal Kac--Moody group} over $k$. Thus $\cG$ is an amalgamated product of a {\bf torus} $T:=\Hom_{\mathrm{gr}}(\Lambda,k^{\times})$ exponentiating the Cartan subalgebra $\hh$, and of the {\bf real root groups} $U_{\alpha}\cong (k,+)$ ($\alpha\in\Delta^{re}$) exponentiating the real root spaces $\g_{\alpha}$. 
The Weyl group $\WW=\WW_A$ can be lifted to a subgroup $N$ of $\cG$ such that $N/T\cong\WW$. For each $\alpha\in\Delta^{re}$, there is a representative $\widetilde{r}_{\alpha}\in N$ of $r_{\alpha}$ such that
\begin{equation}\label{eqn:r_alpha_KM}
\widetilde{r}_{\alpha}\in U_{\alpha}U_{-\alpha}U_{\alpha},
\end{equation}
which we fix.
For any representative $\tilde{w}\in N$ of $w\in \WW$, we also have
\begin{equation}\label{eqn:wUalphawinv_KM}
\tilde{w}U_{\alpha}\tilde{w}\inv=U_{w\alpha}\quad\textrm{for all $\alpha\in\Delta^{re}$.}
\end{equation}
The subgroup $\cU^+$ of $\cG$ generated by all $U_{\alpha}$ with $\alpha\in\Delta^{re+}$ is normalised by $T$ (and intersects $T$ trivially), and $(\mathcal B^+:=T\cU^+,N)$ is a BN-pair for $\cG$. We denote by $X_+$ the associated (locally finite) building. The kernel of the action map $\rho\co \cG\to\Aut(X_+)$ is the centre $Z=\Z(\cG)\subseteq T$ of $\cG$. 

For a subset $J\subseteq I$, we define the subgroups
$$\cL_J^+:=\la U_{\alpha} \ | \ \alpha\in\Delta^{re}(J)\ra, \quad \cU^+(J):=\cL^+_J\cap \cU^+=\la U_{\alpha} \ | \ \alpha\in\Delta^{re+}(J)\ra$$ and $\cL_J:=T\cdot \cL_J^+$
of $\cG$. Note that $\cL_J^+$ contains representatives for each of the elements of $\WW_J$ by (\ref{eqn:r_alpha_KM}). We also let $\cU_J$ denote the normal closure in $\cU^+$ of the subgroup $\la U_{\alpha} \ | \ \alpha\in\Delta^{re+}_J\ra\subseteq \cU^+$. Following \cite[6.2.2]{theseBR}, the standard parabolic subgroup $\mathcal P_J$ of $\cG$ of type $J$ admits a semi-direct decomposition
\begin{equation}\label{eqn:Levidecmin}
\mathcal P_J=\cL_J\ltimes\cU_J.
\end{equation}

The (effective) {\bf geometric completion} of $\cG$ is the closure $G$ of $\cG/Z\cong\rho(\cG)$ in $\Aut(X_+)$, where $\Aut(X_+)$ is equipped with the topology of uniform convergence on bounded sets. In particular, $G$ possesses a basis of identity neighbourhoods consisting of open normal subgroups of $U^+:=\overline{\rho(\cU^+)}$ (namely, the pointwise stabilisers in $U^+$ of balls around the fundamental chamber $C_0$ of $X_+$). Under this topology, any sequence of root groups $(U_{\gamma_n})_{n\in\bN}$, $\gamma_n\in\Delta^{re}$ (where we identify a root group $U_{\gamma}$ with its image in $G$) such that $\height(\gamma_n)\stackrel{n\to\infty}{\to}\infty$ uniformly converges to $\{1\}$ in $G$ (see \cite[Lemma~7]{CaRe} or \cite[Proposition~8.96]{KMGbook}).

For a subset $J\subseteq I$, we consider the closed subgroups 
$$L_J^+:=\overline{\rho(\cL_J^+)}, \quad L_J:=\overline{\rho(\cL_J)}, \quad U^+(J):=\overline{\rho(\cU^+(J))} \quad\textrm{and}\quad U_J:=\overline{\rho(\cU_J)}$$
of $G$. We further consider the closed subgroup $\cH:=\rho(T)$ of $G$ (note that $T$, and hence also $\mathcal H$, is finite). Note that, since $\rho$ is injective on $\cU^+$, we can identify $\cU^+$ (resp. $\cU^+(J)$ or $\cU_J$) with a subgroup of $U^+$ (resp. $U^+(J)$ or $U_J$). As before, identifying $N$ with its (discrete and hence closed) image in $G$, the couple $(\cH U^+,N)$ is a BN-pair for $G$, with same associated building $X_+$.  In particular, $G$ has standard parabolic subgroups $P_J =\overline{\rho(\mathcal P_J)}$ for $J \subseteq I$.

\begin{lemma}\label{lemma:Levi_dec_geo}
Let $J\subseteq I$. Then the following assertions hold:
\begin{enumerate}
\item
$U_J$ is a compact normal subgroup of $P_J$, and $P_J=L_J\cdot U_J$.
\item
$L_J^+$ is normal in $L_J$, and $L_J=\mathcal H\cdot L_J^+$.
\item
If $J_1,\dots,J_n$ are the components of $J$, then each $L^+_{J_i}$ is normal in $L^+_J$ and $L^+_J=L^+_{J_1}\cdot\dots\cdot L^+_{J_n}$.
\item
$\mathcal L_J=\G_{\mc{D}(J)}(k)$, where $\mc{D}(J):=(J,A_J,\Lambda,(c_i)_{i\in J},(h_i)_{i\in J})$.
\item
$L_{J^\infty}^+\cdot U_J$ has finite index in $P_J$.
\item
The center of $L_J$ is contained in $\mathcal H\cdot U^+$.
\end{enumerate}
\end{lemma}
\begin{proof}
(1) and (2) follow from \cite[Lemma~3.1]{openKM}. As the real root groups $U_{\alpha}$ and $U_{\beta}$ commute if $\supp(\alpha)$ and $\supp(\beta)$ lie in different components of $J$, the groups $L_{J_i}^+$ ($i=1,\dots,n$) pairwise commute, yielding (3). The statement (4) holds by definition. Since $P_J=(\mathcal H\cdot L^+_{J^{\sph}}\cdot L^+_{J^{\infty}})\cdot U_J$ by (1), (2) and (3), and $\mathcal H\cdot L^+_{J^{\sph}}=L_{J^{\sph}}$ is finite by (4), the statement (5) follows as well. Finally, the center $Z_J$ of $L_J$ (or equivalently, of $\cL_J$) is the kernel of the action of the Kac--Moody group $\cL_J$ (cf. (4)) on its associated building $X^+_J$, which is embedded in $X^+$ as the standard $J$-residue of $X^+$; in particular, $Z_J$ fixes the fundamental chamber $C_0$ of $X^+$, and hence is contained in its stabiliser $\mathcal H\cdot U^+$ in $G$.
\end{proof}

\subsection{Properties of $G$ as a \tdlc group}\label{subsection:POGAATDLCG}

We conclude this preliminary section on Kac--Moody groups by recording a few known properties of $G$ as a \tdlc group.

\begin{prop}\label{prop:KM_simple}
Assume that $A$ is indecomposable and of non-finite type.  Then $G$ is a compactly generated, nondiscrete \tdlc group. Moreover, the subgroup $G^{(1)}$ of $G$ topologically generated by the $U_{\alpha}$ ($\alpha\in\Delta^{re}$) is a finite index normal subgroup of $G$ which is topologically simple, and we have $G =\cH\cdot G^{(1)}$.
\end{prop}
\begin{proof}
See e.g. \cite[Proposition~8.17]{KMGbook} for the first statement, and \cite[Lemma~9 and Proposition~11]{CaRe} for the second.
\end{proof}
Note that $G=G^{(1)}$ as soon as the Kac--Moody root datum $\mathcal D$ is chosen to be \emph{coadjoint} (see \cite[Definition~7.9 and Example~7.25]{KMGbook}), for instance when $G$ is of simply connected type (see \cite[Example~7.11]{KMGbook}). In general, one can associate to $\mathcal D$ the corresponding coadjoint Kac--Moody root datum $\mathrm{coad}(\mathcal D)$ (see \cite[Exercise~7.14]{KMGbook}), and $G^{(1)}$ is then the (effective) geometric completion of the minimal Kac--Moody group $\mathfrak G_{\mathrm{coad}(\mathcal D)}(k)$ of type $\mathrm{coad}(\mathcal D)$.

The open subgroups of $G$ have been classified up to finite index.

\begin{lemma}[{\cite[Theorem~3.3]{openKM}}]\label{lemma:openKM}
Let $O$ be an open subgroup of $G$. Then there exist $g\in G$, some essential subset $J\subseteq I$, and some spherical subset $J'\subseteq J^\perp$ such that $L^+_J U_{J\cup J^\perp}\subseteq gOg\inv\subseteq P_{J\cup J'}$. In particular, $gOg\inv$ has finite index in $P_{J\cup J'}$.
\end{lemma}

We also record for future reference the following lemma from \cite[Lemma~3.19]{openKM}: the statement below is a bit different from \cite[Lemma~3.19]{openKM}, the subgroup $O_1$ from Lemma~3.19 in  \emph{loc. cit.} being replaced by a subgroup $O$ satisfying the only two properties of $O_1$ (recalled at the beginning of the proof of \cite[Lemma~3.19]{openKM}) that are used in the proof of that lemma.
\begin{lemma}\label{lemma:Jprime}
Let $J\subseteq I$ be essential, and let $O$ be a subgroup of $G$ contained in $P_J$ and containing $L_J^+$. Then every subgroup $H$ of $G$ containing $O$ as a subgroup of finite index
is contained in some standard parabolic $P_{J\cup J'}$ of type $J\cup J'$, with $J'$ spherical and $J'\subseteq J^\perp$.
\end{lemma}

Finally, we recall that, in many cases, $G$ is locally finitely generated, as the subgroup $U^+$ is topologically finitely generated (cf. Lemma~\ref{lem:fqt_local}(1)).

\begin{lemma}\label{lemma:loc_fintype}
The following assertions hold:
\begin{enumerate}
\item
Assume that $p>M_A:=\max_{i\neq j}|a_{ij}|$. Then $G$ is locally finitely generated.
\item
Assume that $A$ is $2$-spherical, i.e. $a_{ij}a_{ji}\leq 3$ for all distinct $i,j\in I$. Assume, moreover, that $q\geq 3$ if $M_A=2$ and that $q\geq 4$ if $M_A=3$. Then $G$ is locally finitely generated. 
\end{enumerate}
\end{lemma}
\begin{proof}
(1) follows from \cite[Theorem~2.2]{RCap} (see also \cite[Proposition~6.11]{Rousseau}), and (2) from \cite[Corollary]{AbrM97}.
\end{proof}

\begin{remark}
Note that, since $U^+$ is pro-$p$ (see e.g. \cite[Proposition~8.17]{KMGbook}), Lemma~\ref{lem:fqt_local} implies that $G$ is locally finitely generated if and only if it is locally of finite quotient type.
\end{remark}

\section{Conjugacy classes of parabolic subgroups up to finite index}\label{sec:conjugacy_parabolic}

Given a Coxeter system $(\WW,S)$, let $\Lambda_{(\WW,S)}$ be the set of standard parabolic subgroups of $\WW$ ordered by inclusion, and let $\Lambda^{\infty}_{(\WW,S)}$ be the same poset taking subgroups up to finite index, that is, the quotient poset $\Lambda_{(\WW,S)}/\sim_f$, where for two subgroups $H_1,H_2$ of a group $H$, we write $H_1\sim_fH_2$ if $H_1\cap H_2$ has finite index in both $H_1$ and $H_2$. We denote by $[H]$ the class of $H\in \Lambda_{(\WW,S)}$ in $\Lambda^{\infty}_{(\WW,S)}$.

Given a group $G$ with a BN-pair and Weyl group $(\WW,S)$, the conjugacy classes of parabolic subgroups of $G$ (up to finite index) are accounted for by elements of $\Lambda_{(\WW,S)}$ (or $\Lambda^{\infty}_{(\WW,S)}$).  In this section, we show that, if $B$ is commensurated (that is, $B \sim_f gBg\inv$ for all $g \in G$), then the conjugation action of $G$ respects the partial order of $\Lambda^{\infty}_{(\WW,S)}$; in other words, if one standard parabolic is virtually contained in a conjugate of another standard parabolic, the standard parabolics were already ordered in this way up to finite index.  In the case of complete geometric Kac--Moody groups, this will lead to a refinement of Lemma~\ref{lemma:openKM}.

\begin{theorem}\label{thm:building_commensurate}
Let $G$ be a group with a BN-pair $(B,N)$ with Weyl group $(\WW,S)$, such that $B$ is a commensurated subgroup of $G$ and $|S| < \infty$; given $J \subseteq S$, let $P_J$ be the standard parabolic subgroup $B\WW_JB$ of $G$.  Let $J,J' \subseteq S$.

Suppose $g \in G$ is such that $gP_{J}g\inv$ is virtually contained in $P_{J'}$.  Then $[\WW_{J}] \le [\WW_{J'}]$ as elements of $\Lambda^{\infty}_{(\WW,S)}$; we have $[\WW_J] = [\WW_{J'}]$ if and only if $gP_Jg\inv \sim_f P_{J'}$.  Conversely, if $[\WW_J] \le [\WW_{J'}]$ then $P_J$ is virtually contained in $P_{J'}$.
\end{theorem}

We begin the proof with two lemmas.  The conjugation action of $\WW$ itself on $\Lambda^{\infty}_{(\WW,S)}$ is accounted for by Lemma~\ref{lem:Deo}.

\begin{lemma}\label{lem:parabolic_inclusion}
Let $(\WW,S)$ be a Coxeter system.  Then the following are equivalent, for $J,J' \subseteq S$:
\begin{enumerate}[(i)]
\item $[\WW_J]\leq [\WW_{J'}]$;
%\item $\WW_J$ is virtually contained in $\WW_{J'}$;
\item $J^{\infty} \subseteq (J')^{\infty}$;
\item There is $w \in \WW$ such that $wJ'w\inv \cap J$ generates a subgroup of finite index in $\WW_J$.
\end{enumerate}
\end{lemma}

\begin{proof}
Suppose (i) holds, so that $\WW_J$ is contained in a subgroup $H$ of $W$ such that $H\cap\WW_{J'}$ has finite index in $H$. Then $\WW_J\cap\WW_{J'}$ has finite index in $\WW_J$, and hence $\WW_J$ is virtually contained in $\WW_{J'}$. Choosing $w \in \WW_J$ such that $\Pc(w^n) = \WW_{J^\infty}$ for all $n \neq 0$ as in Lemma~\ref{lemma:prelim_parabolic}, we then have $w^n \in \WW_{J'}$ for some $n \neq 0$, so $\WW_{J^{\infty}} \le \WW_{J'}$ and hence $J^{\infty} \subseteq J'$.  As the components of $J^{\infty}$ cannot be contained in spherical components of $J'$, we conclude that $J^{\infty} \subseteq (J')^{\infty}$.  Thus (i) implies (ii).

If (ii) holds, then clearly (iii) holds with $w=1$, as $\WW_{J^{\infty}}$ has finite index in $\WW_J$.

Finally, if (iii) holds, then $W_{wJ'w\inv\cap J^{\infty}}$ has finite index in $W_{J^{\infty}}$ and hence $J^{\infty}\subseteq wJ'w\inv$ by Lemma~\ref{lemma:prelim_parabolic}. In particular, $w\inv J^{\infty}w\subseteq J'\subseteq S$ and hence Lemma~\ref{lem:Deo} yields $w\inv J^{\infty}w=J^{\infty}$. Thus,  $J^{\infty} \subseteq J'$, so $\WW_{J'}$ contains the finite index subgroup $\WW_{J^{\infty}}$ of $\WW_J$, showing that (i) holds.
\end{proof}

\begin{corollary}\label{cor:parabolic_inclusion}
Let $(\WW,S)$ be a Coxeter system, let $\Lambda^{\infty}_{(\WW,S)}$ be the set of standard parabolic subgroups of $\WW$ ordered by inclusion, modulo finite index, and let $E$ be the set of essential subsets of $S$ ordered by inclusion.  Then the following map is a well-defined order isomorphism:
\[
\Lambda^{\infty}_{(\WW,S)} \rightarrow E ; \quad [\WW_J] \mapsto J^\infty.
\]
\end{corollary}
We now obtain some conditions under which one parabolic subgroup of a group with a BN-pair is virtually contained in another.

\begin{lemma}\label{lem:parallel}
Let $G$ be a group with a BN-pair with Weyl group $(\WW,S)$, such that $B$ is a commensurated subgroup of $G$; let $X$ be the associated building.
\begin{enumerate}[(i)]
\item Given $J,J' \subseteq S$, then $P_J$ is virtually contained in $P_{J'}$ if and only if $[\WW_J] \le [\WW_{J'}]$.
\item Let $R$ and $R'$ be a pair of parallel residues in $X$.  Then $\mathrm{Stab}_G(R) \sim_f \mathrm{Stab}_G(R')$.
\end{enumerate}
\end{lemma}

\begin{proof}
(i)
We can write $P_J=B\WW_JB$ as a product $P_J = P_{J^{\sph}}P_{J^{\infty}}$ and similarly for $P_{J'}$.  Since $B$ is commensurated, $P_{J^{\sph}} = B\WW_{J^{\sph}}B$ is a union of finitely many left cosets of $B \le P_{J^{\infty}}$, so $J^{\sph}$ does not contribute to the commensurability class of $P_J$.  Thus we may assume $J$ and $J'$ are essential.  We then see by the Bruhat decomposition that if $P_J$ is virtually contained in $P_{J'}$, then $[\WW_J] \le [W_{J'}]$; conversely if $[\WW_J] \le [\WW_{J'}]$, then $J \subseteq J'$ by Lemma~\ref{lem:parabolic_inclusion}, and hence $P_J \le P_{J'}$.

(ii)
Let $P = \mathrm{Stab}_G(R)$ and $P' = \mathrm{Stab}_G(R')$, let $J$ be the type of $R$ and let $J'$ be the type of $R'$.  By \cite[Proposition~21.10]{MPW15}, the fact that $R$ and $R'$ are parallel implies that the Weyl distance $w$ from $R$ to $R'$ satisfies $J = wJ'w\inv$.  In particular, by Lemma~\ref{lem:Deo} we have $w \in \WW_{(J^{\infty})^\perp}$ and $J^{\infty} = (J')^{\infty}$.  Take a residue $R^{\infty}$ of type $J^{\infty}$ in $R$, and let $(R')^{\infty}:=\proj_{R'}R^{\infty}$. Thus, $R^{\infty}$ and $(R')^{\infty}$ are parallel residues of type $J^{\infty}$ contained in $R$ and $R'$, respectively, and the Weyl distance from $R^{\infty}$ to $(R')^{\infty}$ is still $w$ (see \cite[Remark~21.12]{MPW15}).  Moreover, $\mathrm{Stab}_G(R^{\infty})$ has finite index in $P$ and $\mathrm{Stab}_G((R')^{\infty})$ has finite index in $P'$ by (i).  Thus we may assume $J = J^{\infty} = J'$, that is, that $R$ and $R'$ are of the same type.

Since $w \in \WW_{J^\perp}$, we see that $R$ and $R'$ are both contained in a residue $R''$ of type $J \cup J^\perp$.  As a building, we can write $R'' = Y \times Z$ where $Y$ is a building of type $J$ and $Z$ is a building of type $J^\perp$, so that $R = Y \times \{z\}$ and $R' = Y \times \{z'\}$ for some chambers $z,z' \in Z$.  Since $B$ is commensurated, we see that $X$ is locally finite, so $P$ has finite orbits on $Z$ (since it fixes $z \in Z$), and hence $P$ is virtually contained in $P'$; similarly $P'$ is virtually contained in $P$.
\end{proof}

With these lemmas in hand, we can finish the proof of the theorem.

\begin{proof}[Proof of Theorem~\ref{thm:building_commensurate}]
Suppose $J,J' \subseteq S$ and $g \in G$ are such that $gP_{J}g\inv$ is virtually contained in $P_{J'}$; we aim to show $[\WW_{J}] \le [\WW_{J'}]$.  Since $P_{J^\infty}$ has finite index in $P_J$ (and similarly for $P_{J'}$) by Lemma~\ref{lem:parallel}(i), we are free to assume $J$ and $J'$ are essential.  By Lemma~\ref{lemma:prelim_parabolic} we can then take an element $v \in \WW_J$ such that $\Pc(v^n) = \WW_J$ for all $n \in \bN^*$.

Let $X$ be the building associated to the BN-pair.  Then $P_J$ and $P_{J'}$ are the stabilisers of the standard residues $R_J$ and $R_{J'}$ in $X$ of type $J$ and $J'$, respectively.  Let $A$ be an apartment containing a chamber of $gR_J$ and a chamber of $R_{J'}$.  Then we can regard $A$ as a Coxeter complex for the copy of $W$ induced by $\mathrm{Stab}_G(A)$; in particular, there is some element $v'$ of $\mathrm{Stab}_G(A) \cap gP_Jg\inv$ that acts as $v$ on $A$.  By replacing $v'$ by a positive power we may assume that $v' \in P_{J'}$.  By \cite[Proposition~21.8]{MPW15}, $R_1 = \mathrm{proj}_{gR_{J}}(R_{J'})$ is a subresidue of $gR_J$ of type $J \cap wJ'w\inv$, where $w$ is the Weyl distance from $R_1$ to $\proj_{R_{J'}}(gR_J)$.  Since $v'$ stabilises $gR_J$ and $R_{J'}$, it stabilises $R_1$. On the other hand, as $\Pc(v)=\WW_J$, it does not stabilise any proper subresidue of $gR_J$, and hence $R_1=gR_J$ and $J \cap wJ'w\inv = J$.  In particular, we have $[W_{J}] \le [W_{J'}]$ by Lemma~\ref{lem:parabolic_inclusion}, as desired. 

If, in addition, $gP_Jg\inv \sim_f P_{J'}$ (so that $P_{J'}$ is also virtually contained in $gP_Jg\inv$), we have just showed that $[\WW_{J'}] \leq [\WW_{J}]$, and hence $[\WW_J] = [\WW_{J'}]$. Conversely, if $[\WW_J] = [\WW_{J'}]$, then $gP_Jg\inv \sim_f P_{J'}$: indeed, we may again assume $J,J'$ to be essential, so that $J=J'$ by Lemma~\ref{lem:parabolic_inclusion}. Then $\mathrm{proj}_{R_{J'}}(gR_J)=R_{J'}$ (since it is a subresidue of $R_{J'}=R_J$ of type $J'\cap w\inv J w=J$ by \cite[Proposition~21.8]{MPW15}), and hence $R_{J'}$ and $gR_J$ are parallel and the claim follows from Lemma~\ref{lem:parallel}(ii).

Finally, given $J,J' \subseteq S$ such that $[\WW_J] \le [\WW_{J'}]$, then $P_J$ is virtually contained in $P_{J'}$ by Lemma~\ref{lem:parallel}(i).
\end{proof}

In particular, Theorem~\ref{thm:building_commensurate} applies to the parabolic subgroups of complete geometric Kac--Moody groups, and hence via Lemma~\ref{lemma:openKM} we obtain a complete classification of which open subgroups can be virtually conjugated inside one another.

\begin{corollary}\label{cor:capmar:exact}
Let $G$ be a complete geometric Kac--Moody group over a finite field.  Then there is a unique surjective map $\lambda: \mc{O}(G) \rightarrow \Lambda^{\infty}_{(\WW,S)}$, with the following properties:
\begin{enumerate}[(i)]
\item For all $J \subseteq S$, then $\lambda(P_J) = [\WW_J]$;
\item Given $H,K \in \mc{O}(G)$, then $\lambda(H) = \lambda(K)$ if and only if there is $g \in G$ such that $gHg\inv \sim_f K$;
\item Given $H,K \in \mc{O}(G)$, then $\lambda(H) \subsetneq \lambda(K)$ if and only if there is $g \in G$ such that $gHg\inv \sim_f  L$ for a subgroup $L$ of $K$ of infinite index.
\end{enumerate}
\end{corollary}

\section{Locally normal subgroups and ends of Kac--Moody groups}\label{sec:mainKM}

We now study the structure of locally normal subgroups of the geometric completion $G$ of the minimal Kac--Moody group $\cG=\G_{\mc{D}}(k)$ over the finite field $k=\bF_q$ of characteristic $p$.  In this context, we obtain strong restrictions on the noncompact locally normal subgroups of $G$, and a characterisation of when $G$ is one-ended.  We can then apply Theorem~\ref{thm:cpctend:bis} to show in many cases that $G$ is locally indecomposable.  We will use the notation of Sections~\ref{subsubsection:KMRS}--\ref{subsection:POGAATDLCG}.

%%%%%%%%%%%%%%%%%%%%%%%%%%%%%%%%%%%%%%%%%%%%%%%%%%%%%
%%%%%%%%%%%%%%%%%%%%%%%%%%%%%%%%%%%%%%%%%%%%%%%%%%%%%
%%%%%%%%%%%%%%%%%%%%%%%%%%%%%%%%%%%%%%%%%%%%%%%%%%%%%

\subsection{$J$-regular elements}
Since $\WW$ is a finitely generated linear group over $\bC$, Selberg's lemma provides a torsion-free finite-index normal subgroup $\WW_0$ of $\WW$, which we fix throughout. Recall that an element $w\in\WW$ is {\bf straight} if $\ell(w^n)=n\ell(w)$ for all $n\in\bN^*$.
\begin{definition}
 For a subset $J\subseteq I$, call an element $w\in\WW$ {\bf $J$-regular} if 
 $w\in\WW_0$, $w$ is straight, $\Pc(w)=\WW_J$, and $w^n\alpha\neq\alpha$ for all $\alpha\in\Delta^{re+}(J)$ and all $n\in\bN^*$.
\end{definition}

In this subsection and the next, we obtain lower bounds on the closure of the Tits cores of parabolic subgroups, and hence of all open subgroups of $G$.  When $J$ is essential, it will turn out that the closure of the Tits core of the parabolic subgroup $P_J$ is topologically generated by the contraction group of a $J$-regular element and its inverse: see Lemma~\ref{lemma:wdagger_PJdagger} below.  Moreover, we find that every non-compact locally normal subgroup lies between a parabolic subgroup and the closure of its Tits core: see Theorem~\ref{thm:struct_ncpct_lnsbgr}.

Let us first show that $J$-regular elements exist.

\begin{lemma}\label{lemma:existence_J_regular}
Let $J\subseteq I$ be essential. Then there exists a $J$-regular element $w\in\WW$.
\end{lemma}
\begin{proof}
Clearly, there is no loss of generality in assuming that $J$ is irreducible and nonempty. Note also that if $w\alpha\neq\alpha$ for some $w\in\WW_0$ and $\alpha\in \Delta^{re+}(J)$, then also $w^n\alpha\neq\alpha$ for all $n\in\bN^*$ (see \cite[Lemma~2.6]{openKM}). Finally, for any $w\in\WW_J$ of infinite order, there exist some $N\in\bN$ and some $v\in\WW_J$ such that $vw^Nv\inv$ is a straight element of $\WW_0$ (see \cite[Corollary~3.5 and Lemma~4.5]{straight}). In view of Lemma~\ref{lemma:prelim_parabolic}, it is thus sufficient to prove that there exists some $w\in\WW$ with $\Pc(w)=\WW_J$ such that $w\alpha\neq\alpha$ for all $\alpha\in\Delta^{re+}(J)$.

Assume first that $\WW_J$ is an affine Coxeter group. Then the Davis complex $\Sigma_J$ of $(\WW_J,J)$ is a Euclidean space triangulated by the hyperplanes $\partial\alpha$ ($\alpha\in\Delta^{re}$), and there exists some $N\in\bN$ such that for any $v\in\WW_J$, either $v^N=1$ or $v^N$ acts on $\Sigma_J$ as a translation (see \cite[Chapter~10]{BrownAbr}). Note also that $\Pc(v)=\WW_J$ for any $v\in\WW_J$ of infinite order (because $\WW_J$ has no proper infinite parabolic subgroups). One may thus take any $v\in\WW_J$ of infinite order such that the translation axes of $v^N$ are not parallel to any wall of $\Sigma_J$, and set $w:=v^N$.

Assume next that $\WW_J$ is not an affine Coxeter group. Take any $w\in\WW_0$ such that $\Pc(w)=\WW_J$ (such a $w$ exists by \cite[Corollary~2.17]{openKM} and Lemma~\ref{lemma:prelim_parabolic}). Assume for a contradiction that there exists some $\alpha\in\Delta^{re+}(J)$ with $w\alpha=\alpha$. Then $w$ stabilises the wall $\partial\alpha$, and hence possesses an axis $L\subseteq \partial\alpha$. In particular, $\partial\alpha$ intersects all the $w$-essential walls (that is, the walls of a $w$-essential root). On the other hand, for any $w$-essential wall $m$, the walls $w^nm$ ($n\in\bZ$) are pairwise parallel by \cite[Lemma~2.6]{openKM}. Moreover, for any $n\in\bN$ with $|n|$ large enough, the parabolic closure of $r_m$ and $r_{w^nm}$ coincides with $\Pc(w)=\WW_J$ by \cite[Corollary~2.12]{openKM}. Since $\WW_J$ is not an affine Coxeter group, it then follows from \cite[Lemma~11 and Proposition~16]{Cap06} that $r_{\alpha}$ commutes with $r_m$. As $\Pc(w)=\WW_J$ is generated by the reflections $r_m$ with $m$ an $w$-essential wall by \cite[Lemma~2.7]{openKM}, we conclude that $r_{\alpha}$ centralises $\WW_J$ and hence belongs to $\WW_{J^\perp}$ by Lemma~\ref{lem:Deo}, a contradiction since $J\cap J^\perp=\varnothing$.
\end{proof}

The condition of being straight already puts significant restrictions on the real roots fixed by $w$, as the next lemma shows.

\begin{lemma}\label{lemma:real_contracted}
Let $J\subseteq I$ be essential. Let $w\in\WW_0$ be straight and such that $\Pc(w)=\WW_J$. Let $\alpha\in\Delta^{re+}$. Then one of the following holds:
\begin{enumerate}
\item
$w^n\alpha\neq\alpha$ for all $n\in\bN^*$.
\item
$\supp(\alpha)\subseteq J\cup J^\perp$.
\end{enumerate}
\end{lemma}
\begin{proof}
Assume that $w^n\alpha=\alpha$ for some $n\in\bN^*$. Thus $w^n$ commutes with the reflection $r_{\alpha}$ associated to $\alpha$. In particular, $r_{\alpha}$ normalises the parabolic closure $\Pc(w^n)$ of $w^n$. As $\Pc(w^n)=\WW_J$ by Lemma~\ref{lemma:prelim_parabolic}, we deduce that $r_{\alpha}\in N_{\WW}(\WW_J)=\WW_{J\cup J^\perp}$ (recall Lemma~\ref{lem:Deo}). Equivalently, $r_{\alpha}=vs_iv\inv$ for some $v,s_i\in \WW_{J\cup J^\perp}$, and hence $\supp(\alpha)=\supp(v\alpha_i)\subseteq J\cup J^\perp$.
\end{proof}

\subsection{The relative Tits core of locally normal subgroups}

\begin{lemma}\label{lemma:translation_contracted}
Let $w\in\WW$ be straight and $\tilde{w}\in G$ be a representative of $w$ in $N$. Let $\alpha\in\Delta^{re+}$ be such that $w^n\alpha\neq\alpha$ for all $n\in\bN^*$. Then $U_{\alpha}\subseteq \con_{G}(\tilde{w})\cup\con_{G}(\tilde{w}\inv)$. 
\end{lemma}
\begin{proof}
By \cite[Lemma~3.2]{simpleKM}, there exists $\epsilon\in\{\pm1\}$ such that $w^{\epsilon n}\alpha\in\Delta^{re+}$ for all $n\in\bN$. Then $\height(w^{\epsilon n}\alpha)\stackrel{n\to\infty}{\to}\infty$ as $n\to\infty$ since there are only finitely many positive roots of any given height. Hence $$\tilde{w}^{\epsilon n}U_{\alpha}\tilde{w}^{-\epsilon n}=U_{w^{\epsilon n}\alpha}\stackrel{n\to\infty}{\to}\{1\}\quad\textrm{as $n\to\infty$}$$
by (\ref{eqn:wUalphawinv_KM}) in \S\ref{subsubsection:KMG}, that is, $U_{\alpha}\subseteq \con_{G}(\tilde{w}^{\epsilon})$.
\end{proof}

\begin{prop}\label{prop:LJ}
Let $J\subseteq I$ be essential. Let $w\in\WW$ be a $J$-regular element and $\tilde{w}\in G$ a representative of $w$ in $N$. Then $L^+_J\subseteq \overline{G^{\dagger}_{\tilde{w}}}$.
\end{prop}
\begin{proof}
Without loss of generality, we may assume that $J\neq\varnothing$ (so that $w$ has infinite order).
Set
$$\Psi_J:=\{\alpha\in\Delta^{re}(J) \ | \ \textrm{$\alpha$ is $w$-essential}\}.$$
Then $\WW_J=\Pc(w)=\la r_{\alpha} \ | \ \alpha\in\Psi_J\ra$ by \cite[Lemma~2.7]{openKM}. On the other hand, if $\alpha\in\Psi_J$ (so that $w^n\alpha\neq\alpha$ for all $n\in\bN^*$), Lemma~\ref{lemma:translation_contracted} implies that $U_{\pm\alpha}\subseteq G^{\dagger}_{\tilde{w}}$, and hence $\widetilde{r}_{\alpha}\in G^{\dagger}_{\tilde{w}}$ (by (\ref{eqn:r_alpha_KM})). Thus $G^{\dagger}_{\tilde{w}}$ contains the subgroup $\widetilde{\WW}_J:=\langle \widetilde{r}_{\alpha} \ | \ \alpha\in\Psi_J\rangle$ mapping onto $\WW_J$. 

Similarly, $U_{\alpha}\subseteq G^{\dagger}_{\tilde{w}}$ for all $\alpha\in\Delta^{re+}(J)$ by Lemma~\ref{lemma:translation_contracted} (as $w$ is $J$-regular). Since $\mc{L}_J^+$ is generated by such $U_{\alpha}$'s and by $\widetilde{\WW}_J$ (see (\ref{eqn:wUalphawinv_KM})), we conclude that $L_J^+\subseteq \overline{G^{\dagger}_{\tilde{w}}}$, as desired.
\end{proof}

\begin{lemma}\label{lemma:wdagger_PJdagger}
Let $J\subseteq I$ be essential. Let $w\in\WW$ be a $J$-regular element and $\tilde{w}\in G$ a representative of $w$ in $N$. Then $\overline{G^\dagger_{\tilde{w}}}= \overline{(P_J)^\dagger}$.
\end{lemma}
\begin{proof}
Let $H = \overline{G^\dagger_{\tilde{w}}}$.  Since $\tilde{w}\in P_J$ and $P_J$ is open, we have $\con_G(\tilde{w}^{\pm 1})=\con_{P_J}(\tilde{w}^{\pm 1})$ and hence $H = \ol{(P_J)^\dagger_{\tilde{w}}}$.  On the other hand, by Proposition~\ref{prop:LJ}, $H$ contains the subgroup $L^+_J$ of $P_J$, which is cocompact in $P_J$ by Lemma~\ref{lemma:Levi_dec_geo}(1,2).  Thus $H$ is cocompact in $P_J$.  It now follows by Proposition~\ref{prop:dagger_gg} (applied to $G:=P_J$ and $g:=\tilde{w}$) that $(P_J)^\dagger_{\tilde{w}} = (P_J)^\dagger$, and hence $H=\ol{(P_J)^\dagger_{\tilde{w}}}=\overline{(P_J)^\dagger}$, as desired.
\end{proof}

We now arrive at the key theorem of this subsection, putting a lower bound on the closure of the relative Tits core of a $J$-regular element.

\begin{theorem}\label{thm:rTc_w_J}
Let $J\subseteq I$ be essential. Let $w\in\WW$ be a $J$-regular element and $\tilde{w}\in G$ a representative of $w$ in $N$. Then 
\[
L^+_J U_{J\cup J^\perp}\subseteq \overline{G^{\dagger}_{\tilde{w}}}  = \overline{(P_J)^\dagger}.
\]

\end{theorem}
\begin{proof}
Set $H_J:=\{u\tilde{w}u\inv \ | \ u\in \cU^+(J)\}$, so that $\overline{G^{\dagger}_{\tilde{w}}}\subseteq \overline{G^{\dagger}_{H_J}}$. Since $H_J\subseteq P_J$ and $P_J$ is open, we also have $\overline{G^{\dagger}_{H_J}}\subseteq\overline{G^{\dagger}_{P_J}}=\overline{(P_J)^\dagger}$, and hence $\overline{G^{\dagger}_{H_J}}=\overline{G^{\dagger}_{\tilde{w}}}=\overline{(P_J)^\dagger}$ by Lemma~\ref{lemma:wdagger_PJdagger}. To conclude the proof, we now show that $L^+_J U_{J\cup J^\perp}\subseteq \overline{G^{\dagger}_{H_J}}$.

We already know by Proposition~\ref{prop:LJ} that $L^+_J\subseteq\overline{G^{\dagger}_{\tilde{w}}}= \overline{G^{\dagger}_{H_J}}$. By definition of $U_{J\cup J^\perp}$, the inclusion will then follow if we show that for each $\alpha\in\Delta^{re+}_{J\cup J^\perp}$ and each $u\in \cU^+$, there is some $u_J\in \cU^+(J)$ such that $uU_{\alpha}u\inv\subseteq \con_G(u_J\tilde{w}^{\varepsilon}u_J\inv)$ for some $\varepsilon\in\{\pm 1\}$. 
Since $w^n\alpha\neq\alpha$ for all $n\in\bN^*$ by Lemma~\ref{lemma:real_contracted}, we have $U_{\alpha}\subseteq\con(\tilde{w}^{\varepsilon})$ for some $\varepsilon\in\{\pm 1\}$ by Lemma~\ref{lemma:translation_contracted}. Write $u=u_J\overline{u}$ with $u_J\in \cU^+(J)$ and $\overline{u}\in \cU_J$ following the decomposition (\ref{eqn:Levidecmin}) in \S\ref{subsubsection:KMG}. We claim that $\overline{u}U_{\alpha}\overline{u}\inv\subseteq \con_G(\tilde{w}^{\varepsilon})$, whence the lemma. 

Indeed, note that $\tilde{w}$ normalises $\mathcal U_J$, as $\tilde{w}\in \cL_J$ and $\cL_J$ normalises $\mathcal U_J$. In particular, for each $n\in\bN$, there is some $u_n\in \mathcal U_J\subseteq\cU^+$ such that $$\tilde{w}^{\varepsilon n}\overline{u}U_{\alpha}\overline{u}\inv \tilde{w}^{-\varepsilon n}=u_n\tilde{w}^{\varepsilon n}U_{\alpha}\tilde{w}^{-\varepsilon n}u_n\inv.$$ Since $U_{\alpha}\subseteq \con_G(\tilde{w}^{\varepsilon})$ and $G$ has a basis of identity neighbourhoods consisting of normal subgroups of $U^+$ (cf. \S\ref{subsubsection:KMG}), we deduce that $\overline{u}U_{\alpha}\overline{u}\inv\subseteq \con_G(\tilde{w}^{\varepsilon})$ as well, as claimed.
\end{proof}

In light of Lemma~\ref{lemma:openKM}, we obtain a restriction on all Tits cores of open subgroups.

\begin{corollary}\label{cor:open_Tc}

\begin{enumerate}[(i)]
\item Let $J \subseteq I$.  Then
\[
\ol{(P_J)^\dagger} = \ol{(P_{J^\infty})^\dagger} \ge  L^+_{J^{\infty}} U_{J^{\infty} \cup (J^\infty)^\perp}.
\]
\item Let $O$ be an open subgroup of $G$.  Then $\ol{O^\dagger}$ is cocompact in $O$.  In particular, the compact quotient $O/\ol{O^\dagger}$ of $O$ is the largest residually discrete quotient of $O$, so every discrete quotient of $O$ is finite.
\end{enumerate}
\end{corollary}

\begin{proof}
For (i), we note that $P_{J^{\infty}}$ has finite index in $P_J$, so it has the same Tits core.  The conclusion then follows by Theorem~\ref{thm:rTc_w_J}.

For (ii), by Lemma~\ref{lemma:openKM} there is $g \in G$ and a subset $J \subseteq I$ such that $gOg\inv$ is a finite index open subgroup of $P_J$.  Hence, $gO^\dagger g\inv = (P_J)^\dagger$.  From part (i) we see that $\ol{(P_J)^\dagger}$ is a cocompact subgroup of $P_{J^\infty}$, which in turn has finite index in $P_J$.  Hence $\ol{O^\dagger}$ is a cocompact normal subgroup of $O$. In particular, the quotient $O/\ol{O^\dagger}$ is a profinite group, and is thus residually discrete.  On the other hand, if $K$ is an open normal subgroup of $O$, then $K$ is closed and it is clear that $(O/K)^\dagger = \triv$, so $\ol{O^\dagger} \le K$.  Thus the compact quotient $O/\ol{O^\dagger}$ is the largest residually discrete quotient of $O$.  Every discrete quotient of $O$ is therefore compact and hence finite.
\end{proof}

\subsection{Locally normal subgroups}
We are now ready to provide a description of closed locally normal subgroups of $G$ generalising Lemma~\ref{lemma:openKM}.

\begin{lemma}\label{lemma:loc_norm}
Let $K\subseteq I$ and let $O$ be an open finite-index subgroup of $P_K$ containing $L_{K^\infty}^+$. Let $H$ be a closed normal subgroup of $O$ such that $HU_K/U_K$ is infinite. Then there exists some nonempty essential $J\subseteq K$ (a union of essential components of $K$) and some spherical $J'\subseteq J^\perp\cap K$ 
such that 
\[
L^+_J U_{J\cup J^\perp} \le \ol{(P_{J \cup J'})^\dagger} \le H \le P_{J\cup J'}.
\]
\end{lemma}
\begin{proof}
Consider the projection map $\pi\co P_K=L_K U_K\to L_K/(U_K\cap L_K)$. Note first that $\pi(H)$ is closed: indeed, if $(h_n)_{n\in\bN}\subseteq H$ and $g\in P_K$ are such that $\pi(h_n)\stackrel{n\to\infty}{\to}\pi(g)$, then $g\inv h_n$ belongs to $U^+\cdot \ker\pi=U^+$ for all large enough $n$. Since $U^+$ is compact, we may then assume, up to passing to a subsequence, that $h_n$ converges, say to some $h\in H$ (as $H$ is closed). Hence $\pi(g)=\pi(h)\in\pi(H)$, as desired.

Since $\pi(O)\supseteq \pi(L_{K^\infty}^+)$ by assumption, $\pi(H)\cap \pi(L^+_{K^{\infty}})$ is then a closed normal subgroup of $\pi(L^+_{K^{\infty}})$, which is moreover infinite, as $\pi(H)\cong HU_K/U_K$ is infinite and $\pi(H)\cap \pi(L^+_{K^{\infty}})$ has finite index in $\pi(H)\cap\pi(P_K)=\pi(H)$ by Lemma~\ref{lemma:Levi_dec_geo}(5). Let $K_1,\dots,K_n$ be the components of $K^{\infty}$, so that $L^+_{K^\infty}=L^+_{K_1}\cdot\dots\cdot L^+_{K_n}$ by Lemma~\ref{lemma:Levi_dec_geo}(3),  and for each $i$, let $Z_i$ be the center of $L_{K_i}$. By Lemma~\ref{lemma:Levi_dec_geo}(4), the group $L_{K_i}/Z_i$ is the (effective) geometric completion of $\G_{\mc{D}(J_i)}(k)$ (and $Z_i$ is finite), and hence $L^+_{K_i}/Z^+_i$ is topologically simple by Proposition~\ref{prop:KM_simple}, where $Z^+_i:=Z_i\cap L^+_{K_i}$. In particular, $Z^+_{K}:=\prod_{i=1}^nZ^+_i\subseteq L^+_{K^\infty}$ is finite and $L^+_{K^{\infty}}/Z^+_K \cong L^+_{K_1}/Z^+_1\times\dots\times L^+_{K_n}/Z^+_n$ (direct product). Moreover, by simplicity, for each $i$, the image of $\pi(H)\cap \pi(L^+_{K^{\infty}})$ in $\pi(L^+_{K_i})/\pi(Z^+_i)$ is either trivial or coincides with $\pi(L^+_{K_i})/\pi(Z^+_i)$ since $\pi(H)\cap \pi(L^+_{K^{\infty}})$ is normal in $\pi(L^+_{K^{\infty}})$. By Goursat's lemma, a normal subgroup of a direct product of non-abelian simple groups coincides with the product of a subset of the simple direct factors. Therefore $\pi(Z^+_K)\cdot (\pi(H)\cap \pi(L^+_{K^{\infty}}))=\pi(Z^+_K)\cdot\pi(L^+_J)$ for some union $J$ of components of $K^{\infty}$. Note that $J$ is nonempty as $\pi(H)\cap \pi(L^+_{K^{\infty}})$ is infinite. Therefore,
$$L_J^+\subseteq HZ^+_KU_K\quad\textrm{and}\quad HZ^+_KU_K\cap L^+_{K^\infty} U_K\subseteq L^+_JZ^+_K U_K.$$

Moreover, $HZ^+_KU_K\cap L^+_{K^\infty} U_K$ has finite index in $HZ^+_KU_K$ (because $L^+_{K^\infty} U_K$ has finite index in $P_K$ by Lemma~\ref{lemma:Levi_dec_geo}(5)) and $L^+_JZ^+_K U_K\subseteq P_J$ by Lemma~\ref{lemma:Levi_dec_geo}(6). We are thus in a position to apply Lemma~\ref{lemma:Jprime} (with $O:=HZ^+_KU_K\cap L^+_{K^\infty} U_K$ and $H:=HZ^+_KU_K$), which yields some spherical subset $J'\subseteq J^\perp$ such that $HZ^+_KU_K\subseteq P_{J\cup J'}$. In particular, $H\subseteq P_{J\cup J'}$, and since $H\subseteq P_K$, we may assume that $J'\subseteq J^\perp\cap K$.

Let now $w\in\WW_{J}$ be a $J$-regular element (see Lemma~\ref{lemma:existence_J_regular}) and $\tilde{w}\in G$ be a representative of $w$ in $N\cap L_J^+$.  Then $\tilde{w} \in L^+_J \le HZ^+_KU_K$, so by Lemma~\ref{lemma:dagger_cc} (applied to $G:=G$, $H:=HZ^+_KU_K$ and $K:=H$) we have $G_{\tilde{w}}^\dagger \le G^{\dagger}_{HZ^+_KU_K}=G^{\dagger}_H$. Hence Lemma~\ref{lemma:dagger_containment} (applied to $G:=G$, $D=X:=H$, and $U:=O$) yields $G_{\tilde{w}}^\dagger \le G^{\dagger}_H\le H$.
Theorem~\ref{thm:rTc_w_J} then implies that  
\[
L^+_J\cdot U_{J\cup J^\perp} \le \ol{(P_{J \cup J'})^\dagger}=\ol{(P_{J})^\dagger} = \ol{G_{\tilde{w}}^\dagger} \le H.\qedhere
\]
\end{proof}

\begin{theorem}\label{thm:struct_ncpct_lnsbgr}
Let $H$ be a non-compact closed locally normal subgroup of $G$. Then there exist some $g\in G$, some nonempty essential subset $J\subseteq I$ and some spherical subset $J'\subseteq J^\perp$ such that
\[
L^+_J U_{J\cup J^\perp}\le \ol{(P_{J \cup J'})^\dagger} \le gHg\inv \le P_{J\cup J'}
\]
and such that the normaliser of $gHg\inv$ in $P_{J \cup J'}$ has finite index.
\end{theorem}
\begin{proof}
Let $O=N_{G}(H)$. Up to conjugating $H$, we may assume by Lemma~\ref{lemma:openKM} that $O$ has finite index in some standard parabolic $P_K$ ($K\subseteq I$) and contains $L_{K^{\infty}}^+$. Moreover, since $H$ is non-compact, the group $HU_K/U_K$ is infinite. The claim now readily follows from Lemma~\ref{lemma:loc_norm}.
\end{proof}

\begin{corollary}
Assume that $J^\perp$ is spherical for every nonempty essential subset $J\subseteq I$. Then every closed locally normal subgroup of $G$ is either compact or open.
\end{corollary}

%%%%%%%%%%%%%%%%%%%%%%%%%%%%%%%%%%%%%%%%%%%%%%%%%%%%%
%%%%%%%%%%%%%%%%%%%%%%%%%%%%%%%%%%%%%%%%%%%%%%%%%%%%%
\subsection{Centralisers}

We now consider centralisers of locally normal subgroups.  The first point to note is an immediate consequence of Proposition~\ref{prop:KM_simple} and \cite[Theorem~A(ii)]{CRWpart2}.

\begin{lemma}\label{lem:A_semisimple}
Assume that $A$ is indecomposable of non-finite type.  Then every virtually abelian locally normal subgroup of $G^{(1)}$ is trivial.  In particular, $\QZ(G^{(1)}) = \triv$.
\end{lemma}

We can now show that the centraliser of $U_{J \cup J^\perp}$ for $J \subseteq I$ essential is trivial, and consequently the centraliser of any nontrivial locally normal subgroup is compact.

\begin{lemma}\label{lemma:centraliser_Ujjperp}
Assume that $A$ is indecomposable of non-finite type. Let $J\subsetneq I$ be essential and nonempty. Then $\CC_{G^{(1)}}(U_{J \cup J^\perp}) = \triv$.
\end{lemma}
\begin{proof}
Let $Z = \CC_{G^{(1)}}(U_{J \cup J^\perp})$.  Assume for a contradiction that $Z$ is nontrivial, and set $K:=J \cup J^\perp$. Note first that $Z$ is a closed locally normal subgroup, as it is normal in the normaliser $P_K$ of $U_K$.  By Lemma~\ref{lem:A_semisimple} we have
$$Z\cap U_K=\{1\}.$$
Moreover, $Z$ is infinite (if it is finite, then it is virtually abelian and hence trivial, a contradiction). In particular, $ZU_K/U_K$ is infinite. We can thus apply Lemma~\ref{lemma:loc_norm} (with $O:=P_K$ and $H:=Z$): there exists some nonempty essential subset $J_1\subseteq K$ (a union of essential components of $K$) such that $U_{J_1\cup J_1^\perp}\subseteq Z$. Since $K\subseteq J_1\cup J_1^\perp$, so that $U_{J_1\cup J_1^\perp}\subseteq U_K$, we deduce that 
$$U_{J_1\cup J_1^\perp}\subseteq Z\cap U_K=\{1\},$$
that is, $J_1\cup J_1^\perp=I$. The indecomposability of $A$ then implies that $J_1=K=I$ and hence also that $J=\varnothing$ or $J=I$, a contradiction.
\end{proof}

\begin{corollary}\label{cor:Ristcpct}
Assume that $A$ is indecomposable of non-finite type. Let $H$ be a closed locally normal subgroup of $G^{(1)}$, and suppose that $H$ and $\CC_{G^{(1)}}(H)$ are both nontrivial.  Then $H$ and $\CC_{G^{(1)}}(H)$ are both compact.
\end{corollary}
\begin{proof}
Let $K = \CC_{G^{(1)}}(H)$.  Then $H$ and $K$ both have nontrivial centraliser, so by Lemma~\ref{lemma:centraliser_Ujjperp}, they cannot contain any subgroup $U_{J \cup J^\perp}$ (or their conjugates)  for $J \subseteq I$ essential and nonempty.  Hence $H$ and $K$ are both compact by Theorem~\ref{thm:struct_ncpct_lnsbgr} (applied to $G:=G^{(1)}$).
\end{proof}

\subsection{The number of ends}\label{subsec:ends}
We next show that $G$ is one-ended if and only if its Weyl group $\WW$ is one-ended. By \cite[Theorem~8.7.2]{DavisCox}, this latter condition can be reformulated in terms of the \emph{nerve} of $\WW$.

Recall that the \defbold{nerve} of a Coxeter system $(\WW,I)$ is the abstract simplicial complex formed by the poset of all nonempty spherical subsets $J \subseteq I$; write $\mc{N}_A$ for the nerve of the Coxeter system associated to the generalised Cartan matrix $A$.  We say the nerve is \defbold{strongly connected} if in its geometric realisation $|\mc{N}_A|$, the removal of any one of the closed simplices (including the empty simplex) results in a connected space.

This condition can also be restated in terms of the Coxeter diagram of $(\WW,I)$, as follows. 
The {\bf finite graph} $\Gamma_f(A)$ of $A$ is the graph with vertex set $I$ and an edge between $i$ and $j$ if and only if $a_{ij}a_{ji}\leq 3$ (equivalently, if $s_is_j\in\WW$ has finite order). We call $\Gamma_f(A)$ {\bf strongly connected} if it is connected and if for every spherical subset $J\subseteq I$, the subgraph of $\Gamma_f(A)$ with vertex set $I\setminus J$ is still connected. For instance, if $A$ is $2$-spherical and indecomposable, then $\Gamma_f(A)$ is strongly connected.

\begin{prop}\label{prop:one-ended}
Assume that $\WW_A$ is infinite. The following conditions are equivalent:
\begin{enumerate}
\item
$G$ is one-ended.
\item
$\WW_A$ is one-ended.
\item 
$\mc{N}_A$ is strongly connected.
\item
$\Gamma_f(A)$ is strongly connected.
\end{enumerate}
\end{prop}
\begin{proof}
The equivalence of (3) and (4) is clear, as $\Gamma_f(A)$ is just the $1$-skeleton of $\mc{N}_A$, and as a simplicial complex is connected if and only if its $1$-skeleton is connected. 
The equivalence of (2) and (3) is \cite[Theorem~8.7.2]{DavisCox}. For the equivalence of (1) and (2), since $G$ acts geometrically on (the Davis realisation of) its associated building $X_+$ (which is both proper and geodesic, cf. \S\ref{subsection:BNpairsaB}), we see that $G$ is one-ended if and only if $X_+$ is one-ended.  Meanwhile, each apartment of $X_+$ is quasi-isometric to $\WW_A$, so that the equivalence of (1) and (2) follows from Lemma~\ref{lemma:XoneendediffSigmais} below.
\end{proof}

\begin{lemma}\label{lemma:XoneendediffSigmais}
Let $X$ be the Davis realisation of a building. Then $X$ is one-ended if and only if its apartments are one-ended.
\end{lemma}
\begin{proof}
Let $(\WW,S)$ be the type of $X$ (note that we may assume $\WW$ to be infinite). 

Note first that for any apartment $\Sigma$ of $X$ and any point $x\in \Sigma$, there exist a geodesic line $L_x\subseteq \Sigma$ and a (closed) chamber $C_x$ of $\Sigma$ containing $x$ such that $C_x\cap L_x$ is nonempty. Indeed, as $\Sigma$ is isometric to the Davis complex of $(\WW,S)$, we may assume that $X=\Sigma=|\Sigma(\WW,S)|_{\CAT}$. One can then consider any geodesic line $L$ in $\Sigma$ (e.g. the axis of an element of infinite order in $\WW\subseteq \Aut(\Sigma)$) and any chamber $C$ of $\Sigma$ intersecting $L$, and then pick $w\in \WW$ such that $wx\in C$. Then $x$ belongs to the chamber $C_x:=w\inv C$ and $C_x$ intersects the geodesic line $L_x:=w\inv L\subseteq\Sigma$, as desired.

Assume now that the apartments of $X$ are one-ended. Let $r_1\co [0,\infty)\to X$ and $r_2\co [0,\infty)\to X$ be proper geodesic rays. Let $B\subseteq X$ be a ball in $X$. Let $N\in\bN$ be such that all the (closed) chambers intersecting either $R_1:=r_1([N,\infty))$ or $R_2:=r_2([N,\infty))$ are disjoint from $B$. We claim that $R_1$ and $R_2$ are contained in the same path component of $X\setminus B$, as desired. Let $x_i:=r_i(N)\in R_i$ be the origin of $R_i$ ($i=1,2$), and let us show that $x_1$ and $x_2$ are connected in $X\setminus B$. By \cite[Lemma~2.2]{geohyperbolic}, there exists an apartment $\Sigma$ of $X$ containing both $x_1$ and a subray $R_2'=r_2([N',\infty))$ of $R_2$ (for some $N'\geq N$). Let $C_{x_1}$ be a chamber of $\Sigma$ containing $x_1$, and $L_{x_1}\subseteq \Sigma$ be a geodesic line intersecting $C_{x_1}$ (say at $\overline{x}_1\in C_{x_1}\cap L_{x_1}$), as provided by the previous paragraph. By construction, $C_{x_1}$ is disjoint from $B$, and since it is path-connected, it suffices to prove that $\overline{x}_1$ and $x_2$ are connected in $X\setminus B$. Since $\overline{x}_1\notin B$ and $B$ is convex, we find some geodesic ray $R_1'\subseteq L_{x_1}$ with origin $\overline{x}_1$ and not intersecting $B$. Since $\Sigma$ is one-ended, there exist some $x'_1\in R_1'$ and some $x'_2\in R_2'$ connected by a path in $\Sigma\setminus B\subseteq X\setminus B$. Since $\overline{x}_1$ and $x_1'$ (resp. $x_2$ and $x_2'$) are in the same path component of $X\setminus B$, the claim follows.

Conversely, assume that $X$ is one-ended. Let $\Sigma$ be an apartment of $X$, let $r_1\co [0,\infty)\to \Sigma$ and $r_2\co [0,\infty)\to \Sigma$ be proper geodesic rays, and let $B_{\Sigma}\subseteq \Sigma$ be a ball in $\Sigma$ with center $x_0\in \Sigma$ and radius $r$ (say $x_0$ is the barycenter of a chamber $C$ of $\Sigma$). Let $B\subseteq X$ be the ball of $X$ of center $x_0$ and radius $r$, so that $B_{\Sigma}=B\cap \Sigma$. By assumption, there exists some $N\in\bN$ such that $R_1:=r_1([N,\infty))$ and $R_2:=r_2([N,\infty))$ are contained in the same path component of $X\setminus B$. Let $x_1\in R_1$ and $x_2\in R_2$, and let $\Gamma\subseteq X\setminus B$ be a path from $x$ to $y$. Since the \emph{retraction} $\rho=\rho_{\Sigma,C}\co X\to \Sigma$ onto $\Sigma$ centered at $C$ preserves the distances from $x_0$ (see e.g. \cite[Proposition~12.18]{BrownAbr}), it maps $X\setminus B$ to $\Sigma\setminus B_{\Sigma}$, and hence the path $\rho(\Gamma)$ from $x$ to $y$ is contained in $\Sigma\setminus B_{\Sigma}$. Thus, $R_1$ and $R_2$ are contained in the same path component of $\Sigma\setminus B_{\Sigma}$, as desired.
\end{proof}

\subsection{Endgame}

Putting all the ingredients together, we now have a sufficient condition for $G$ (or equivalently its finite index normal subgroup $G^{(1)}$) to be locally indecomposable.

\begin{theorem}\label{thm:LDG}
Assume that $A$ is indecomposable, that $G$ is locally finitely generated, and that $\Gamma_f(A)$ is strongly connected. Then $G$ is locally indecomposable.
\end{theorem}
\begin{proof}
Since $G^{(1)}$ is open in $G$, it is sufficient to show that $G^{(1)}$ is locally indecomposable by Lemma~\ref{lem:loc_indecomp}. Note that we may assume $A$ to be of non-finite type, since otherwise $G^{(1)}$ is finite and hence trivially locally indecomposable. 

We now check that $G^{(1)}$ satisfies the hypotheses of Theorem~\ref{thm:cpctend:bis}. Specifically: $G^{(1)}$ is one-ended by Proposition~\ref{prop:one-ended}; $\QZ(G^{(1)}) = \triv$ by Lemma~\ref{lem:A_semisimple}; $G^{(1)}$ has no nontrivial compact normal subgroups by Proposition~\ref{prop:KM_simple}; $G^{(1)}$ is locally finitely generated by assumption (since it is open in $G$) and hence locally of finite quotient type by Lemma~\ref{lem:fqt_local}(1); no open subgroup of $G^{(1)}$ has an infinite discrete quotient by Corollary~\ref{cor:open_Tc}; and the centraliser of every nontrivial closed locally normal subgroup is compact by Corollary~\ref{cor:Ristcpct}.
\end{proof}

Together with Lemma~\ref{lemma:loc_fintype}, this yields the following corollaries.
\begin{corollary}
Assume $A$ is indecomposable. Assume, moreover, that $\Gamma_f(A)$ is strongly connected, and that $p>M_A$. Then $G$ is locally indecomposable.
\end{corollary}

\begin{corollary}
Assume $A$ is of indecomposable $2$-spherical type. Assume, moreover, that $q\geq 3$ if $M_A=2$ and $q\geq 4$ if $M_A=3$.
Then $G$ is locally indecomposable.
\end{corollary}

\subsection{Open questions}

Corollary~\ref{corintro:localindecKM} shows that a large class of locally compact Kac--Moody groups are locally indecomposable. On the other hand, it might well be that the additional assumptions in Theorem~\ref{intro:KM_loc_indec} can be removed, which leads to the following question.

\begin{question}
Does every complete geometric Kac--Moody group over a finite field have a trivial decomposition lattice?
\end{question}

More generally, as mentioned in the introduction, locally compact Kac--Moody groups might even have trivial \emph{centraliser} lattice. This calls for an answer to the following question.

\begin{question}
Does every complete geometric Kac--Moody group over a finite field have a trivial centraliser lattice?
\end{question}

\bibliographystyle{amsalpha} 
\bibliography{LocalindecKM}

\end{document}